\documentclass{article}
\usepackage{amsmath,amssymb}
\usepackage{amsthm}
\newtheorem{thm}{Theorem}
\newtheorem{defn}[thm]{Definition}
\newtheorem{prop}[thm]{Proposition}
\newtheorem{lem}[thm]{Lemma}
\newtheorem{cor}[thm]{Corollary}

\newtheorem{rem}[thm]{Remark}

\begin{document}
\title{Geometric Properties of Orbits of Hermann actions}
\author{Shinji Ohno}
\date{\today}
\maketitle
\begin{abstract}
In this paper, 
we investigate properties of orbits of Hermann actions as submanifolds
without assuming the commutability of involutions which define Hermann actions.
In particular, we compute the second fundamental form of orbits of Hermann action,
and give a sufficient condition for orbits of Hermann action to be weakly reflective (resp. arid) submanifolds.
\end{abstract}

\tableofcontents
\section{Introduction}
Let $G$ be a compact semisimple Lie group, 
and  $\theta_{1}, \theta_{2}$ be involutive automorphisms of $G$.
For $i=1,2$, $G_{\theta_{i}}$ and $(G_{\theta_{i}})_{0}$ denote the fixed point set of $\theta_{i}$ 
and identity component of $G_{\theta_{i}}$, respectively.
Take a subgroup $K_{i}$ of $G$ which satisfies 
$(G_{\theta_{i}})_{0} \subset K_{i}\subset G_{\theta_{i}}$
for $i=1,2$.
The triple $(G, K_{1}, K_{2})$ is called a compact symmetric triad.
For $i=1,2$, 
$(G, K_{i})$ is a compact symmetric pair.
Thus, the coset manifold $G/K_{i}$ is a compact Riemannian symmetric space with respect to a $G$-invariant Riemannian metric.

The Lie group actions
\begin{itemize}
\item $K_{2} \curvearrowright G/K_{1}: k_{2}(gK_{1})=(k_{2}g)K_{1}\ (k_{2}\in K_{2}, g\in G)$
\item $K_{1} \curvearrowright G/K_{2}: k_{1}(gK_{2})=(k_{1}g)K_{2}\ (k_{1}\in K_{1}, g\in G)$
\end{itemize}
are called Hermann actions.

By definition, a Hermann action is an isometric action on a symmetric space,
and a generalization of the isotropy action of a compact symmetric space.
Actually, in the cases of  $K_{1}=K_{2}$, 
the Hermann action is the isotropy action of the symmetric space $G/K_{1}$.   
The Hermann action is a generalization of the isotropy action in a compact symmetric space.
In previous research,
the second fundamental form and geometrical properties of orbits of the isotropy actions in the compact symmetric spaces are being investigated. (\cite{HTST, Ve}, etc).
The Hermann action is known to be a hyperpolar action similar to the isotropy action of a compact symmetric space.

An isometric action of  a compact Lie group on a Riemannian manifold $M$ is called hyperpolar 
if there exists a closed, connected and flat submanifold $S$ of $M$ 
that meets all orbits orthogonally. 
Then the submanifold $S$ is called a section.

Korlloss(\cite{K1, K2}) shows that
an indecomposable hyperpolar action of cohomogeneity greater than one on a Riemannian symmetric space of compact type is orbit equivalent to a Hermann action.

Matsuki(\cite{M1, M2}) shows that the orbit space of Hermann actions are described by root systems.

In 2007, Goertsches and Thorbergsson (\cite{GT}) showed that the Shape operator and the curvature operator of the orbit of Hermann actions are commutative.
In the case of $ \theta_ {1} \theta_ {2} = \theta_ {2} \theta_ {1} $, the eigenvalues of the shape operator is determined.
Even in the case of $ \theta_ {1} \theta_ {2} \neq \theta_ {2} \theta_ {1} $, the eigenvalue of the shape operator of the orbits with several conditions is calculated.

In 2011, Ikawa(\cite{I1}) introduce the notion of symmetric triads with multiplicities.
In the cases of $\theta_{1}\theta_{2}=\theta_{2}\theta_{1}$, Ikawa gives a  characterization of minimal, austere and totally geodesic orbits of Hermann actions in terms of symmetric triads with multiplicities.  

On the assumption $\theta_{1}\theta_{2}=\theta_{2}\theta_{1}$, 
a sufficient condition for orbits of Hermann action to be weakly reflective and 
a characterization of biharmonic orbits of Hermann actions were given in \cite{Ohno1} and  \cite{OSU1, OSU2}. 

In particular, 
\cite{Ohno1} gives  
weakly reflective orbits of Hermann action which satisfy $\theta_{1}\theta_{2}\neq \theta_{2}\theta_{1}$.

In this paper, 
we investigate properties of orbits of Hermann actions as submanifolds without assuming the condition $\theta_{1}\theta_{2}=\theta_{2}\theta_{1}$.
From Matsuki's classification (\cite{M2}), we can see that $\theta_ {1} $ and $\theta_ {2} $ can be replaced so that $ \theta_ {1} \theta_ {2} $ has a finite order.
Therefore, in this paper, we assume that the order of $ \theta_ {1} \theta_ {2} $ is finite.

This article is organized as follows.
In Section~\ref{sect-preliminaries}, 
we prepare root systems which associated with compact symmetric triads.
In particular, in Lemma~\ref{onb}, we prepare orthonormal bases for the calculation of the second fundamental forms of orbits of Hermann actions. 
In Section~\ref{sect-orbitspace}, 
we describe the orbit spaces of the Hermann actions using the root system prepared in Section~\ref{sect-preliminaries}.
In Section~\ref{sect-2nd}, we compute the second fundamental form of orbits of Hermann action. 
Moreover, we give a characterization of minimal, austere and totally geodesic orbits of Hermann actions. 
In Section~\ref{sect-wr}, we give a sufficient condition for orbits of Hermann action to be weakly reflective (resp. arid) submanifolds. 
In Section~\ref{list}, 
we introduce the examples of Hermann actions whose involutions are not commutative.
In section~\ref{classify}, 
we classify austere orbits and weakly reflective orbits of Hermann actions which are introduced in Section~\ref{list}.

\section{Preliminaries}\label{sect-preliminaries}
Let $G$ be a compact,  semisimple, and connected Lie group, 
and $\theta_{1}, \theta_{2}$ be involutive automorphisms of $G$.
For $i=1, 2$, we take and fix a subgroup $K_{i}$ of $G$ which satisfies 
$(G_{\theta_{i}})_{0} \subset K_{i}\subset G_{\theta_{i}}$.
Then the triple $(G, K_{1}, K_{2})$ is a compact symmetric triad. 

We denote the Lie algebras of $G$, $K_{1}$ and $K_{2}$ by $\mathfrak{g}$, $\mathfrak{k}_{1}$ and $\mathfrak{k}_{2}$, respectively.
Fix an $\mathrm{Ad}(G)$-invariant inner product $\langle , \rangle$ on $\mathfrak{g}$.
Then the coset manifolds $G/K_{1}=:M_{1}$ and $K_{2}\backslash G=:M_{2}$ are  compact Riemannian symmetric spaces with respect to $G$-invariant Riemannian metric which are induced by $\langle , \rangle$. 
For $i=1, 2$, the involutive automorphism of  $\mathfrak{g}$ which is induced from $\theta_{i}$ will be also denoted by $\theta_{i}$.
Then, $\mathfrak{k}_{i}=\{X \in \mathfrak{g} \mid \theta_{i}(X)=X\}$ holds. 
Now, we have two orthogonal direct products of $\mathfrak{g}$:
\begin{align*}
\mathfrak{g}
=\mathfrak{k}_{1}\oplus \mathfrak{m}_{1}
=\mathfrak{k}_{2}\oplus \mathfrak{m}_{2}
\end{align*}
where 
$\mathfrak{m}_{i}=\{X \in \mathfrak{g} \mid \theta_{i}(X)=-X\}$.
For $i=1, 2$, we denote by $\pi_{i}$ the natural projection from $G$ onto $M_{i}$ 
(i.e. $\pi_{1}(g)=gK_{1} \in M_{1}, \ \pi_{2}(g)=K_{2}g \in M_{2}$ for $g \in G$).

Hereafter, we suppose  
$(\theta_{1}\theta_{2})^{l}=\mathrm{id}_{G}$ for some natural number $l$.
Since $\theta_{1}$ and $\theta_{2}$ are involutions, 
when $l=1$, $\theta_{1}=\theta_{2}$ and
when $l=2$, $\theta_{1}\theta_{2}=\theta_{2}\theta_{1}$ hold.
The eigenvalues of $\theta_{1}\theta_{2}$ belong to the unitary group $\mathrm{U}(1)$, since $\theta_{1}\theta_{2}$ is an automorphism of $\mathfrak{g}$.
We denote the complexification of $\mathfrak{g}$ by  $\mathfrak{g}^{\mathbb{C}}$.
Then we have the eigenspace decomposition of $\mathfrak{g}^{\mathbb{C}}$:
\begin{align*}
\mathfrak{g}^{\mathbb{C}}=\sum_{\varepsilon \in \mathrm{U}(1)} \mathfrak{g}_{\varepsilon}
\end{align*}
where 
$\mathfrak{g}_{\varepsilon}=\{ X \in \mathfrak{g}^{\mathbb{C}} \mid \theta_{1}\theta_{2} (X)=\varepsilon X\}$.
For $\varepsilon \in \mathrm{U}(1)$ satisfying $\mathfrak{g}_{\varepsilon}\neq \{0 \}$,  $\varepsilon^{l}=1$ holds.
For an eigenvalue $\varepsilon$ of $\theta_{1}\theta_{2}$, 
we define $\varphi_{\varepsilon} \in ( -\pi /2, \pi/2 ]$ by the equation $\varepsilon =e^{2\sqrt{-1} \varphi_{\varepsilon}}$.

We can see that 
\begin{align*}
\mathfrak{g}_{1}\cap \mathfrak{g}=\mathfrak{k}_{1}\cap\mathfrak{k}_{2}\oplus \mathfrak{m}_{1}\cap\mathfrak{m}_{2}.
\end{align*}
Moreover
$\mathfrak{g}_{1}\cap \mathfrak{g}$ is an orthogonal symmetric Lie algebra with respect to $\theta_{1}$ and $\theta_{2}$.
We remark that $\mathfrak{g}_{1}\cap \mathfrak{g}$ is not necessarily semisimple.

Take and fix a maximal abelian subspace $\mathfrak{a}$ of $\mathfrak{m}_{1}\cap\mathfrak{m}_{2}$.
We consider the adjoint representation of $\mathfrak{a}$ on $\mathfrak{g}^{\mathbb{C}}$.
For each $\alpha \in \mathfrak{a}$,
we set 
\begin{align*}
\mathfrak{g}(\alpha)=\{ X \in \mathfrak{g}^{\mathbb{C}} \mid [H, X] =\sqrt{-1} \langle \alpha, H\rangle X \quad (H \in \mathfrak{a}) \}
\end{align*}
and $\tilde{\Sigma}=\{ \alpha \in \mathfrak{a} \setminus \{0\} \mid \mathfrak{g}(\alpha) \neq \{0 \}\}$.
Then, we have a direct sum decomposition of $\mathfrak{g}^{\mathbb{C}}$:
\begin{align*}
\mathfrak{g}^{\mathbb{C}}=\mathfrak{g}(0) \oplus \sum_{\alpha \in \tilde{\Sigma} }\mathfrak{g}(\alpha). 
\end{align*}

For each $\alpha \in \tilde{\Sigma} \cup \{0\}$, 
$\theta_{1}\theta_{2}(\mathfrak{g}(\alpha)  )=\mathfrak{g}(\alpha)  $ holds. 
We set  
\begin{align*}
\mathfrak{g}(\alpha, \varepsilon)=\{ X \in \mathfrak{g}(\alpha) \mid \theta_{1}\theta_{2}(X) = \varepsilon X\}
\end{align*}
for $\alpha \in \tilde{\Sigma} \cup \{0\}$ and $\varepsilon \in \mathrm{U}(1)$.
Then we have 
\begin{align*}
\mathfrak{g}(\alpha)=\sum_{\varepsilon \in \mathrm{U}(1)} \mathfrak{g}(\alpha, \varepsilon).
\end{align*}
Let $\overline{X}$ denotes the complex conjugate of  $X \in \mathfrak{g}^{\mathbb{C}}$.
Since
$\overline{\mathfrak{g}(\alpha)}=\mathfrak{g}(-\alpha)$, 
if $\alpha \in \tilde{\Sigma} $, then $-\alpha \in \tilde{\Sigma} $ holds. 
Moreover, 
\cite{M1} shows the following lemma.  
\begin{lem}[\cite{M1}]\label{lem-M-bracket}
For $\alpha, \beta \in \mathfrak{a}, \ \varepsilon , \delta \in \mathrm{U}(1)$, 
\begin{enumerate}
\item $[\mathfrak{g}(\alpha) , \mathfrak{g}(\beta)] \subset \mathfrak{g}(\alpha+\beta)$
\item $\theta_{1} ( \mathfrak{g}(\alpha))= \mathfrak{g}(-\alpha) ,\ \theta_{2}(\mathfrak{g}(\alpha))= \mathfrak{g}(-\alpha)$
\item $\theta_{1}\theta_{2}(\mathfrak{g}(\alpha))= \mathfrak{g}(\alpha),\ \theta_{2}\theta_{1}(\mathfrak{g}(\alpha))= \mathfrak{g}(\alpha)$
\item $\theta_{1}(\mathfrak{g}(\alpha, \varepsilon))=\mathfrak{g}(-\alpha, \varepsilon^{-1})$,\ 
$\theta_{2}(\mathfrak{g}(\alpha, \varepsilon))=\mathfrak{g}(-\alpha, \varepsilon^{-1})$
\item $\overline{\mathfrak{g}(\alpha, \varepsilon)}=\mathfrak{g}(-\alpha, \varepsilon^{-1})$
\item $[\mathfrak{g}(\alpha, \varepsilon), \mathfrak{g}(\beta, \delta)]\subset \mathfrak{g}(\alpha+ \beta, \varepsilon \delta)$
\end{enumerate}
\end{lem}
Moreover, \cite{I1} shows the following lemma. 
\begin{lem}[\cite{I1}]\label{lem-I-root}
The set $\tilde{\Sigma}$ is a root system of $\mathfrak{a}\cap \mathfrak{z}^{\perp}$,  
where 
$\mathfrak{z}$ is the center of $\mathfrak{g}$. 
\end{lem}

For each $\varepsilon \in \mathrm{U}(1)$, 
we set 
$\Sigma_{\varepsilon}=\{\alpha \in \mathfrak{a}\setminus \{ 0\} \mid \mathfrak{g}(\alpha, \varepsilon)\neq \{0\} \}$.
Then, from Lemma~\ref{lem-M-bracket}, 
we have  $-\alpha \in \Sigma_{\varepsilon^{-1}}$ for each $\alpha \in \Sigma_{\varepsilon}$.
Since $\mathfrak{g}(\alpha, \varepsilon)\subset \mathfrak{g}(\alpha)$, 
we can see that 
$\Sigma_{\varepsilon} \subset \tilde{\Sigma}$
and 
$\tilde{\Sigma} = \bigcup_{\varepsilon \in \mathrm{U}(1)} \Sigma_{\varepsilon}$.

We define 
\begin{align*}
m(\alpha, \varepsilon):=\dim_{\mathbb{C}} \mathfrak{g}(\alpha, \varepsilon), 
m(\alpha):=\sum_{\varepsilon \in \mathrm{U}(1)} m(\alpha, \varepsilon) =\dim_{\mathbb{C}} \mathfrak{g}(\alpha) 
\end{align*}
for $\alpha \in \tilde{\Sigma}, \varepsilon \in \mathrm{U}(1)$. 
Then,  we get 
$m(\alpha, \varepsilon)=m(-\alpha, \varepsilon^{-1}),\ m(\alpha)=m(-\alpha)$
from
Lemma~\ref{lem-M-bracket}. 

We take a fundamental system $\tilde{\Pi} =\{ \alpha_{1}, \ldots , \alpha_{r}\}$ of $\tilde{\Sigma}$.
We denote by $\tilde{\Sigma}^{+}$ the set of positive roots in $\tilde{\Sigma}^{+}$.
Set 
$\Sigma_{\varepsilon}^{+}=\Sigma_{\varepsilon}\cap \tilde{\Sigma}^{+}$
for $\varepsilon \in \mathrm{U}(1)$.
Then 
\begin{align}\label{poralization}
&\mathfrak{g}(\alpha, \varepsilon)\oplus \mathfrak{g}(-\alpha, \varepsilon^{-1})\nonumber \\
=&\{ X \in \mathfrak{g}_{\varepsilon}\oplus \mathfrak{g}_{\varepsilon^{-1}} 
\mid [H, [H, X]] =-\langle \alpha, H\rangle^{2} X \quad (H \in \mathfrak{a})\}\\
=&\{ X \in \mathfrak{g}_{\varepsilon}\oplus \mathfrak{g}_{\varepsilon^{-1}} 
\mid [H, [H', X]] =-\langle \alpha, H\rangle \langle \alpha, H'\rangle X \quad (H , H' \in \mathfrak{a})\}\nonumber
\end{align}
holds (see  \cite{I1}). 
Hence, 
\begin{align*}
&(\mathfrak{g}(\alpha, \varepsilon)\oplus \mathfrak{g}(-\alpha, \varepsilon^{-1})) \cap \mathfrak{g}\\
=&\{ X \in (\mathfrak{g}_{\varepsilon}\oplus \mathfrak{g}_{\varepsilon^{-1}})\cap \mathfrak{g} 
\mid [H, [H, X]] =-\langle \alpha, H\rangle^{2} X \quad (H \in \mathfrak{a})\}
\end{align*}
holds. 
We set 
$V(\alpha, \varepsilon)=(\mathfrak{g}(\alpha, \varepsilon)\oplus \mathfrak{g}(-\alpha, \varepsilon^{-1})) \cap \mathfrak{g}$.
Since 
$\dim V(\alpha, \varepsilon)=m(\alpha, \varepsilon)$, 
we have the following  direct sum decomposition:
\begin{align*}
\mathfrak{g}=(\mathfrak{g}(0)\cap \mathfrak{g})
\oplus \sum_{\varepsilon \in \mathrm{U}(1)} \sum_{\alpha \in \Sigma_{\varepsilon }^{+}} V(\alpha, \varepsilon).
\end{align*}
More precisely, 
when we set 
$V(0, \varepsilon)=(\mathfrak{g}(0, \varepsilon)\oplus \mathfrak{g}(0, \varepsilon^{-1})) \cap \mathfrak{g}$ for 
$\varepsilon \in \mathrm{U}(1)$, 
we have 
\begin{align*}
\mathfrak{g}(0)\cap \mathfrak{g}=V(0,1)\oplus V(0, -1)\oplus \sum_{\varepsilon \in \mathrm{U}(1) \atop \mathrm{Im}(\varepsilon) >0} V(0, \varepsilon).
\end{align*}

Since 
$\theta_{i}(V(\alpha, \varepsilon))=V(\alpha, \varepsilon)$
for $i=1, 2$, 
we have the direct sum decompositions
\begin{align*}
V(\alpha, \varepsilon)=(V(\alpha, \varepsilon)\cap \mathfrak{k}_{1})\oplus(V(\alpha, \varepsilon) \cap \mathfrak{m}_{1})
=(V(\alpha, \varepsilon)\cap \mathfrak{k}_{2})\oplus(V(\alpha, \varepsilon) \cap \mathfrak{m}_{2}).
\end{align*}
Then we have the following lemma.
\begin{lem}\label{onb}
For each $\varepsilon \in \mathrm{U}(1)$ and $\alpha \in \Sigma_{\varepsilon}^{+}$, 
there exist orthonormal bases 
$\{X_{\alpha, i}^{\varepsilon}\}_{i=1}^{m(\alpha, \varepsilon)}$ and  
$\{Y_{\alpha, i}^{\varepsilon}\}_{i=1}^{m(\alpha, \varepsilon)}$
of $V(\alpha, \varepsilon)\cap \mathfrak{k}_{1}$ and 
$V(\alpha, \varepsilon)\cap \mathfrak{m}_{1}$
respectively such that 
satisfy
the following two conditions;
\begin{enumerate}
\item[(1)] For any $H \in \mathfrak{a}$
\begin{align*}
[H, X_{\alpha, i}^{\varepsilon}] =\langle\alpha , H \rangle Y_{\alpha, i}^{\varepsilon},\ 
[H, Y_{\alpha, i}^{\varepsilon}] =-\langle\alpha , H \rangle X_{\alpha, i}^{\varepsilon},\ 
[X_{\alpha, i}^{\varepsilon}, Y_{\alpha, i}^{\varepsilon}]=\alpha ,
\end{align*}
\begin{align*}
\mathrm{Ad}(\exp H) X_{\alpha, i}^{\varepsilon}&= \cos \langle \alpha, H \rangle X_{\alpha, i}^{\varepsilon}+ \sin \langle \alpha, H\rangle Y_{\alpha, i}^{\varepsilon}, \\
\mathrm{Ad}(\exp H) Y_{\alpha, i}^{\varepsilon}&= -\sin \langle \alpha, H \rangle X_{\alpha, i}^{\varepsilon}+ \cos \langle \alpha, H\rangle Y_{\alpha, i}^{\varepsilon}
\end{align*}
holds.
\end{enumerate}
\item[(2)] For $H_{\varepsilon}\in \mathfrak{a}$ which satisfies $\langle \alpha ,H_{\varepsilon}\rangle= \varphi_{\varepsilon}$, 
we set 
$\widetilde{X}_{\alpha, i}^{\varepsilon}=\mathrm{Ad}(\exp H_{\varepsilon})^{-1}X_{\alpha, i}^{\varepsilon}$ and 
$\widetilde{Y}_{\alpha, i}^{\varepsilon}=\mathrm{Ad}(\exp H_{\varepsilon})^{-1}Y_{\alpha, i}^{\varepsilon}$. 
Then 
$\{ \widetilde{X}_{\alpha, i}^{\varepsilon} \}_{i=1}^{m(\alpha, \varepsilon)}$ and  
$\{ \widetilde{Y}_{\alpha, i}^{\varepsilon}\}_{i=1}^{m(\alpha, \varepsilon)}$
are orthonormal bases of $V(\alpha, \varepsilon)\cap \mathfrak{k}_{2}$ and 
$V(\alpha, \varepsilon)\cap \mathfrak{m}_{2}$
respectively, 
and 
for any $H \in \mathfrak{a}$
\begin{align*}
[H, \widetilde{X}_{\alpha, i}^{\varepsilon}] =\langle\alpha , H \rangle \widetilde{Y}_{\alpha, i}^{\varepsilon},\ 
[H, \widetilde{Y}_{\alpha, i}^{\varepsilon}] =-\langle\alpha , H \rangle \widetilde{X}_{\alpha, i}^{\varepsilon},\ 
[\widetilde{X}_{\alpha, i}^{\varepsilon}, \widetilde{Y}_{\alpha, i}^{\varepsilon}]=\alpha ,
\end{align*}
\begin{align*}
\mathrm{Ad}(\exp H) \widetilde{X}_{\alpha, i}^{\varepsilon}&= \cos \langle \alpha, H \rangle \widetilde{X}_{\alpha, i}^{\varepsilon}+ \sin \langle \alpha, H\rangle \widetilde{Y}_{\alpha, i}^{\varepsilon}, \\
\mathrm{Ad}(\exp H) \widetilde{Y}_{\alpha, i}^{\varepsilon}&= -\sin \langle \alpha, H \rangle \widetilde{X}_{\alpha, i}^{\varepsilon}+ \cos \langle \alpha, H\rangle \widetilde{Y}_{\alpha, i}^{\varepsilon}
\end{align*}
holds.
\end{lem}
\begin{proof}
For $(1)$, 
let $\{X_{\alpha, i}^{\varepsilon}\}_{i=1}^{m(\alpha, \varepsilon)}$ be an 
orthonormal basis of $V(\alpha, \varepsilon)\cap \mathfrak{k}_{1}$.
We set 
\begin{align*}
Y_{\alpha, i}^{\varepsilon}=\left[ \frac{\alpha}{\langle \alpha, \alpha \rangle }, X_{\alpha, i}^{\varepsilon} \right]
\end{align*}
for $1 \leq i\leq m(\alpha, \varepsilon)$.
Then we have 
$Y_{\alpha, i}^{\varepsilon}\in V(\alpha, \varepsilon)\cap \mathfrak{m}_{1}$
and
\begin{align*}
\langle Y_{\alpha, i}^{\varepsilon}, Y_{\alpha, j}^{\varepsilon}\rangle 
=&\left\langle \left[ \frac{\alpha}{\langle \alpha, \alpha \rangle }, X_{\alpha, i}^{\varepsilon} \right], \left[ \frac{\alpha}{\langle \alpha, \alpha \rangle }, X_{\alpha, j}^{\varepsilon} \right] \right\rangle \\
=&-\left\langle  X_{\alpha, i}^{\varepsilon} , \frac{1}{\langle \alpha, \alpha \rangle^{2} } [\alpha , [\alpha , X_{\alpha, j}^{\varepsilon} ]] \right\rangle \\
=&\langle  X_{\alpha, i}^{\varepsilon} ,  X_{\alpha, j}^{\varepsilon} \rangle 
=\delta_{i,j}
\end{align*}
for $1\leq i, j \leq m(\alpha,\varepsilon)$.
Thus 
$\{Y_{\alpha, i}^{\varepsilon}\}_{i=1}^{m(\alpha, \varepsilon)}$ is an 
orthonormal basis of $V(\alpha, \varepsilon)\cap \mathfrak{m}_{1}$.
From equation (\ref{poralization})
we have 
\begin{align*}
[H, Y_{\alpha, i}^{\varepsilon}]=\left[ H , \left[ \frac{\alpha }{\langle \alpha, \alpha \rangle }, X_{\alpha, i}^{\varepsilon} \right] \right]= -\langle \alpha, H\rangle X_{\alpha, i}^{\varepsilon}
\end{align*}
for $H\in \mathfrak{a}$.
Thus $X_{\alpha, i}^{\varepsilon}= -\left[ \frac{\alpha }{\langle \alpha, \alpha \rangle } ,Y_{\alpha, i}^{\varepsilon} \right]$ 
holds. 
Hence, we also have 
\begin{align*}
[H, X_{\alpha, i}^{\varepsilon}]=-\left[H , \left[ \frac{\alpha }{\langle \alpha, \alpha \rangle }, Y_{\alpha, i}^{\varepsilon }\right]\right]=-\langle \alpha, H  \rangle Y_{\alpha, i}^{\varepsilon} 
\end{align*}
for $H\in \mathfrak{a}$.
Since 
$V(\alpha, \varepsilon)= (\mathfrak{g}(\alpha, \varepsilon)\oplus \mathfrak{g}(-\alpha, \varepsilon^{-1}))\cap \mathfrak{g}$, 
and $\overline{\mathfrak{g}(\alpha, \varepsilon)}=\mathfrak{g}(-\alpha, \varepsilon^{-1})$, 
there exists $X\in \mathfrak{g}(\alpha, \varepsilon)$, such that $X_{\alpha, i}^{\varepsilon} =X+\overline{X}$.
Then 
\begin{align*}
[X_{\alpha, i}^{\varepsilon}, Y_{\alpha, i}^{\varepsilon}]
=&\left[ (X+\overline{X}) \left[ \frac{\alpha }{\langle \alpha, \alpha \rangle },(X+\overline{X}) \right]\right]\\
=&[X+\overline{X}, \sqrt{-1} X-\sqrt{-1}\ \overline{X}]
=-2\sqrt{-1}[X, \overline{X}].
\end{align*}
Therefore, 
\begin{align*}
\theta_{1}\theta_{2}([X_{\alpha, i}^{\varepsilon}, Y_{\alpha, i}^{\varepsilon}])
=&\theta_{1}\theta_{2}(-2\sqrt{-1}[X, \overline{X}])
=-2\sqrt{-1}[\varepsilon X, \varepsilon^{-1} \overline{X}]\\
=&-2\sqrt{-1}[X, \overline{X}]
=[X_{\alpha, i}^{\varepsilon}, Y_{\alpha, i}^{\varepsilon}].
\end{align*}
Obviously, 
\begin{align*}
\theta_{1}[X_{\alpha, i}^{\varepsilon}, Y_{\alpha, i}^{\varepsilon}]=-[X_{\alpha, i}^{\varepsilon}, Y_{\alpha, i}^{\varepsilon}]
\end{align*}
holds.
Thus, 
$\theta_{2}([X_{\alpha, i}^{\varepsilon}, Y_{\alpha, i}^{\varepsilon}])
=-\theta_{2}\theta_{1}([X_{\alpha, i}^{\varepsilon}, Y_{\alpha, i}^{\varepsilon}])
=-[X_{\alpha, i}^{\varepsilon}, Y_{\alpha, i}^{\varepsilon}].
$
Hence, we get 
$[X_{\alpha, i}^{\varepsilon}, Y_{\alpha, i}^{\varepsilon}] \in \mathfrak{m}_{1}\cap \mathfrak{m}_{2}$.
Furter, since
\begin{align*}
[H, [X_{\alpha, i}^{\varepsilon}, Y_{\alpha, i}^{\varepsilon}]]
=&-[X_{\alpha, i}^{\varepsilon}, [Y_{\alpha, i}^{\varepsilon}, H]]- [Y_{\alpha, i}^{\varepsilon}, [H, X_{\alpha, i}^{\varepsilon}]]\\
=&-\langle \alpha, H \rangle [X_{\alpha, i}^{\varepsilon}, X_{\alpha, i}^{\varepsilon}] 
-\langle \alpha, H \rangle [Y_{\alpha, i}^{\varepsilon}, Y_{\alpha, i}^{\varepsilon}]=0, 
\end{align*}
 we have $[X_{\alpha, i}^{\varepsilon}, Y_{\alpha, i}^{\varepsilon}] \in \mathfrak{a}$.
Additionally, for $H \in \mathfrak{a}$, 
\begin{align*}
\langle H, [X_{\alpha, i}^{\varepsilon}, Y_{\alpha, i}^{\varepsilon}]\rangle
=\langle [H, X_{\alpha, i}^{\varepsilon}], Y_{\alpha, i}^{\varepsilon}\rangle 
=\langle \alpha, H\rangle \langle Y_{\alpha, i}^{\varepsilon}, Y_{\alpha, i}^{\varepsilon}\rangle 
=\langle \alpha, H\rangle
\end{align*}
holds.
Therefore, we obtain
$[X_{\alpha, i}^{\varepsilon}, Y_{\alpha, i}^{\varepsilon}]=\alpha$.

For (2), 
since $\alpha \neq 0$, 
There exists $H_{\varepsilon} \in \mathfrak{a}$ such that $\langle \alpha , H_{\varepsilon}\rangle =\varphi_{\varepsilon}$.
Then 
\begin{align*}
\widetilde{X}_{\alpha, i}^{\varepsilon}
=&\mathrm{Ad}(\exp H_{\varepsilon})^{-1} X_{\alpha, i}^{\varepsilon}=\cos (-\varphi_{\varepsilon}) X_{\alpha, i}^{\varepsilon}+ \sin (-\varphi_{\varepsilon}) Y_{\alpha, i}^{\varepsilon}\\
\widetilde{Y}_{\alpha, i}^{\varepsilon}
=&\mathrm{Ad}(\exp H_{\varepsilon})^{-1} Y_{\alpha, i}^{\varepsilon}=-\sin (-\varphi_{\varepsilon}) X_{\alpha, i}^{\varepsilon}+ \cos (-\varphi_{\varepsilon}) Y_{\alpha, i}^{\varepsilon}
\end{align*}
holds.
Hencce, 
$\widetilde{X}_{\alpha, i}^{\varepsilon}$ and $\widetilde{Y}_{\alpha, i}^{\varepsilon}$ do not depend on $H_{\varepsilon}$.
There exists $X\in \mathfrak{g}(\alpha, \varepsilon)$ such that 
$X_{\alpha, i}^{\varepsilon}=X+\overline{X}$. 
Then 
\begin{align*}
\theta_{2}(\widetilde{X}_{\alpha, i}^{\varepsilon})
=&\theta_{2}(\mathrm{Ad}(\exp H_{\varepsilon})^{-1} X_{\alpha, i}^{\varepsilon})
=\mathrm{Ad}(\exp H_{\varepsilon}) \theta_{2}(X_{\alpha, i}^{\varepsilon})\\
=&\mathrm{Ad}(\exp H_{\varepsilon}) \theta_{2}\theta_{1}(X_{\alpha, i}^{\varepsilon})
=\mathrm{Ad}(\exp H_{\varepsilon}) \theta_{2}\theta_{1}(X+\overline{X})\\
=&\mathrm{Ad}(\exp H_{\varepsilon}) (\varepsilon^{-1} X+\varepsilon \overline{X})
=\mathrm{Ad}(\exp H_{\varepsilon})^{-1}\mathrm{Ad}(\exp H_{\varepsilon})^{2} (\varepsilon^{-1} X+\varepsilon \overline{X})\\
=&\mathrm{Ad}(\exp H_{\varepsilon})^{-1} (e^{2\sqrt{-1}\varphi_{\varepsilon}}\varepsilon^{-1} X+e^{-2\sqrt{-1}\varphi_{\varepsilon}} \varepsilon \overline{X})
=\mathrm{Ad}(\exp H_{\varepsilon})^{-1} ( X+ \overline{X})\\
=&\widetilde{X}_{\alpha, i}^{\varepsilon}
\end{align*}
Thus $\widetilde{X}_{\alpha, i}^{\varepsilon} \in V(\alpha, \varepsilon) \cap \mathfrak{k}_{2}$.
\begin{align*}
\theta_{2}(\widetilde{Y}_{\alpha, i}^{\varepsilon})
=&\theta_{2}(\mathrm{Ad}(\exp H_{\varepsilon})^{-1} Y_{\alpha, i}^{\varepsilon})\\
=&\theta_{2}\left( \mathrm{Ad}\left( \exp H_{\varepsilon} \right)^{-1} \left[\frac{\alpha}{\langle \alpha, \alpha \rangle }, X_{\alpha, i}^{\varepsilon}\right] \right)
=\theta_{2}\left( \left[\frac{\alpha}{\langle \alpha, \alpha \rangle }, \widetilde{X}_{\alpha, i}^{\varepsilon} \right] \right)\\
=&-\left[\frac{\alpha}{\langle \alpha, \alpha \rangle }, \widetilde{X}_{\alpha, i}^{\varepsilon} \right] 
=-\widetilde{Y}_{\alpha, i}^{\varepsilon}.
\end{align*}
Thus, $\widetilde{Y}_{\alpha, i}^{\varepsilon} \in V(\alpha, \varepsilon) \cap \mathfrak{m}_{2}$.
Since $\mathrm{Ad}(\exp H_{\varepsilon})^{-1}$ is an orthogonal transformation and an automorphism on $\mathfrak{g}$, 
$\{ \widetilde{X}_{\alpha, i}^{\varepsilon} \}_{i=1}^{m(\alpha, \varepsilon)}$ and  
$\{ \widetilde{Y}_{\alpha, i}^{\varepsilon}\}_{i=1}^{m(\alpha, \varepsilon)}$
of $V(\alpha, \varepsilon)\cap \mathfrak{k}_{2}$ and 
$V(\alpha, \varepsilon)\cap \mathfrak{m}_{2}$
respectively.
\begin{align*}
[H, \widetilde{X}_{\alpha, i}^{\varepsilon}] =\langle\alpha , H \rangle \widetilde{Y}_{\alpha, i}^{\varepsilon},\ 
[H, \widetilde{Y}_{\alpha, i}^{\varepsilon}] =-\langle\alpha , H \rangle \widetilde{X}_{\alpha, i}^{\varepsilon},\ 
[\widetilde{X}_{\alpha, i}^{\varepsilon}, \widetilde{Y}_{\alpha, i}^{\varepsilon}]=\alpha ,
\end{align*}
\begin{align*}
\mathrm{Ad}(\exp H) \widetilde{X}_{\alpha, i}^{\varepsilon}&= \cos \langle \alpha, H \rangle \widetilde{X}_{\alpha, i}^{\varepsilon}+ \sin \langle \alpha, H\rangle \widetilde{Y}_{\alpha, i}^{\varepsilon}, \\
\mathrm{Ad}(\exp H) \widetilde{Y}_{\alpha, i}^{\varepsilon}&= -\sin \langle \alpha, H \rangle \widetilde{X}_{\alpha, i}^{\varepsilon}+ \cos \langle \alpha, H\rangle \widetilde{Y}_{\alpha, i}^{\varepsilon}
\end{align*}
holds for all $H \in \mathfrak{a}$.
\end{proof}
\begin{rem}\rm
\begin{itemize}
\item Lemma \ref{onb} is a generalization of Lemma 4.16 (2)  in \cite{I1}. 
\item When $G$ is simple and $(\theta_{1}\theta_{2})^{2}=\mathrm{id}_{G}$, 
if $\Sigma_{1}\cap \Sigma_{-1}\neq \emptyset$ holds, 
then $(\tilde{\Sigma}, \Sigma_{1}, \Sigma_{-1})$ is a symmetric triad which is intoduced in \cite{I1}. 
\end{itemize}
\end{rem}

\section{Orbit space}\label{sect-orbitspace}
In this section, 
we describe orbit spaces of Hermann actions.  

In order to consider the orbit space of the action of $K_{2}$ on $M_{1}=G/K_{1}$, 
we define an equivalence relation $\sim$ on $G$.
For $g_{1}, g_{2} \in G$, 
\begin{align*}
g_{1}\sim g_{2} \iff K_{2}\pi_{1}(g_{1})=K_{2}\pi_{1}(g_{2}).
\end{align*}
Then $G/\sim$ is the  orbit space of the action of $K_{2}$ on $M_{1}$, 
and we can identify $G/\sim$ with double coset $K_{2}\backslash G /K_{1}$.

We define a group $\tilde{J}$ by 
\begin{align*}
\tilde{J}=
\{ ([s], Y) \in \mathrm{N}_{K_{2}}(\mathfrak{a}) / \mathrm{Z}_{ K_{1}\cap K_{2}} (\mathfrak{a}) \ltimes \mathfrak{a} \mid \exp(-y)s \in K_{1} \}
\end{align*}
where we set 
\begin{align*}
\mathrm{N}_{L}(\mathfrak{a})=&\{ s\in L \mid \mathrm{Ad}(s) \mathfrak{a}=\mathfrak{a}\},\\
\mathrm{Z}_{L}(\mathfrak{a})=&\{ s\in L \mid \mathrm{Ad}(s)|_{\mathfrak{a}}=\mathrm{id}_{\mathfrak{a}}\}
\end{align*}
for a subgroup $L$ of $G$.
The group $\tilde{J}$ acts on  $\mathfrak{a}$ by  the following:
\begin{align*}
([s], Y)\cdot H =\mathrm{Ad}(s)H+Y\quad (([s], Y) \in \tilde{J} ,\ H \in \mathfrak{a}).
\end{align*}
It is known that the following proposition.
\begin{prop}[\cite{M1}]
\begin{align*}
K_{2}\backslash G /K_{1}\cong \mathfrak{a}/\tilde{J}
\end{align*}
\end{prop}
From this proposition, 
we can see that in order to describe $K_{2}\backslash G /K_{1}$, 
we need to examine the structure of $\tilde{J}$. 

For $\alpha \in \tilde{\Sigma}$, we set 
\begin{align*}
s_{\alpha}(H)=H-2\frac{\langle \alpha, H\rangle }{\langle \alpha, \alpha \rangle}\alpha \quad (H\in \mathfrak{a}).
\end{align*}
Then $s_{\alpha}$ is the reflection in $\mathfrak{a}$ with respect to the hyperplane 
$\{H \in \mathfrak{a} \mid \langle \alpha, H\rangle=0 \}$. 
The next lemma shows relationship between $\tilde{\Sigma}$ and $\tilde{J}$.

\begin{lem}\label{J-tilde}
For $\varepsilon \in \mathrm{U}(1),\  \alpha \in \Sigma_{\varepsilon }^{+},\  n\in \mathbb{Z}$, 
\begin{align*}
\left( s_{\alpha}, 2\frac{n\pi -\varphi_{\varepsilon}}{\langle \alpha, \alpha\rangle } \alpha \right) \in \tilde{J}.
\end{align*}
\end{lem}
\begin{proof}
For $\varepsilon \in \mathrm{U}(1)$ which satisfying $\Sigma_{\varepsilon}^{+}\neq \emptyset $ and $\alpha \in\Sigma_{\varepsilon}^{+}$, 
we take a vector $\widetilde{X}_{\alpha, i}^{\varepsilon}$ in Lemma~\ref{onb}.
Then we have  
$\widetilde{X}_{\alpha, i}^{\varepsilon} \in \mathfrak{k}_{2}$ and  
\begin{align*}
[\widetilde{X}_{\alpha, i}^{\varepsilon}, H]=-\langle \alpha, H \rangle \widetilde{Y}_{\alpha, i}^{\varepsilon},\ \ 
[\widetilde{X}_{\alpha, i}^{\varepsilon}, [\widetilde{X}_{\alpha, i}^{\varepsilon}, H]]=-\langle \alpha , H\rangle \alpha . 
\end{align*}
Thus for any  $t \in \mathbb{R}$ and $H \in \mathfrak{a}$, we have
\begin{align*}
\mathrm{Ad}(\exp (t \widetilde{X}_{\alpha, i}^{\varepsilon}))H
=H- \langle \alpha, H\rangle \left\{ \frac{1}{\| \alpha \|} \sin (t \| \alpha \|) \widetilde{Y}_{\alpha, i}^{\varepsilon} + \frac{1}{\| \alpha \|^{2}} (1- \cos (t \|\alpha \| )) \alpha \right\} .
\end{align*}
In particular, 
\begin{align*}
\mathrm{Ad}\left( \exp \left( \frac{\pi}{\| \alpha \|} \widetilde{X}_{\alpha, i}^{\varepsilon}\right) \right)H
=H- 2 \frac{\langle \alpha, H\rangle}{\langle \alpha, \alpha \rangle }\alpha  =s_{\alpha}(H)
\end{align*}
holds.
Therfore, 
$s_{\alpha } \in \mathrm{N}_{K_{2}}(\mathfrak{a}) $.
We take a vector $H_{\varepsilon}\in \mathfrak{a}$ which satisfy $\langle \alpha, H_{\varepsilon} \rangle = n\pi -\varphi_{\varepsilon}$.
Set $a=\exp(H_{\varepsilon})$. 
Then the element 
\begin{align*}
 \exp \left( \frac{\pi}{\| \alpha \|} \widetilde{X}_{\alpha, i}^{\varepsilon}\right) a^{-1} \left( \exp \left(- \frac{\pi}{\| \alpha \|} \widetilde{X}_{\alpha, i}^{\varepsilon}\right) \right) a
\end{align*}
can be calculated by the following two way.
By using Lemma~\ref{onb}, 
\begin{align*}
 &\exp \left( \frac{\pi}{\| \alpha \|} \widetilde{X}_{\alpha, i}^{\varepsilon}\right) a^{-1} \left( \exp \left(- \frac{\pi}{\| \alpha \|} \widetilde{X}_{\alpha, i}^{\varepsilon}\right) \right) a\\
=& \exp \left( \frac{\pi}{\| \alpha \|} \widetilde{X}_{\alpha, i}^{\varepsilon}\right) \left( \exp \left(- \mathrm{Ad}(a)^{-1}\frac{\pi}{\| \alpha \|} \widetilde{X}_{\alpha, i}^{\varepsilon}\right) \right) \\
=& \exp \left( \frac{\pi}{\| \alpha \|} \widetilde{X}_{\alpha, i}^{\varepsilon}\right) \left( \exp \left(- \frac{\pi}{\| \alpha \|} (-1)^{n} X_{\alpha, i}^{\varepsilon}\right) \right) .
\end{align*}
On the other hand, 
\begin{align*}
 &\exp \left( \frac{\pi}{\| \alpha \|} \widetilde{X}_{\alpha, i}^{\varepsilon}\right) a^{-1} \left( \exp \left(- \frac{\pi}{\| \alpha \|} \widetilde{X}_{\alpha, i}^{\varepsilon}\right) \right) a\\
=& \exp\left(- \mathrm{Ad} \left( \exp \left( \frac{\pi}{\| \alpha \|} \widetilde{X}_{\alpha, i}^{\varepsilon}\right)\right) H_{\varepsilon}\right) a\\
=& \exp\left(- s_{\alpha} (H_{\varepsilon} ) \right) \exp (H_{\varepsilon})\\
=& \exp\left( H_{\varepsilon} -\left( H_{\varepsilon} - 2 \frac{\langle \alpha, H_{\varepsilon} \rangle}{\langle \alpha, \alpha \rangle } \alpha \right)\right) \\
=& \exp\left( 2 \frac{\langle \alpha, H_{\varepsilon} \rangle}{\langle \alpha, \alpha \rangle } \alpha \right) \\
=& \exp\left( 2 \frac{n\pi -\varphi_{\varepsilon }}{\langle \alpha, \alpha \rangle } \alpha \right) \\
\end{align*}
Then we have
\begin{align*}
\exp\left( -2 \frac{n\pi -\varphi_{\varepsilon }}{\langle \alpha, \alpha \rangle } \alpha \right) \left( \frac{\pi}{\| \alpha \|} \widetilde{X}_{\alpha, i}^{\varepsilon}\right) 
=\left( \exp \left(- \frac{\pi}{\| \alpha \|} (-1)^{n} X_{\alpha, i}^{\varepsilon}\right) \right) \in K_{1}.
\end{align*}
Therefore, we can see that 
\begin{align*}
\left( s_{\alpha}, 2\frac{n\pi -\varphi_{\varepsilon}}{\langle \alpha, \alpha\rangle } \alpha \right) \in \tilde{J}.
\end{align*}
\end{proof}

We denote the subgroup of $\mathrm{O}(\mathfrak{a}) \ltimes \mathfrak{a}$ generated by the set 
\begin{align*}
\bigcup_{\varepsilon \in \mathrm{U}(1)} 
\left\{ 
\left. 
\left( s_{\alpha}, 2\frac{ n \pi -\varphi_{\varepsilon } }{\langle \alpha, \alpha \rangle }\right) \
\right| \ \alpha \in \Sigma_{\varepsilon}, n \in \mathbb{Z}
\right\}
\end{align*}
by $W(\tilde{\Sigma}, \{ \Sigma_{\varepsilon} \}_{\varepsilon \in \mathrm{U}(1)} )$. 
From Lemma~\ref{J-tilde}, 
$W(\tilde{\Sigma}, \{ \Sigma_{\varepsilon} \}_{\varepsilon \in \mathrm{U}(1)} )$ is a subgroup of 
$\tilde{J}$.
Matsuki (\cite{M1}, Prop~3.1) proved  that 
 $W(\tilde{\Sigma}, \{ \Sigma_{\varepsilon} \}_{\varepsilon \in \mathrm{U}(1)} )=\tilde{J}$
 when $G$ is simply connected.
If $G$ is not simply connected, 
then the orbit space $K_{2}\backslash G /K_{1}$ can be identified with 
$\overline{P}/ \{\sigma \in \tilde{J} \mid  \sigma \overline{P} = \overline{P}\}$,
where $P$ is a fundamental domain of the action of $W(\tilde{\Sigma}, \{ \Sigma_{\varepsilon} \}_{\varepsilon \in \mathrm{U}(1)} )$ 
on $\mathfrak{a}$ and $\overline{P}$ denotes the closure of $P$.

Hereafter, we consider $\overline{P}$.

The element $\left( s_{\alpha}, 2\frac{ n \pi -\varphi_{\varepsilon } }{\langle \alpha, \alpha \rangle }\right) \in W(\tilde{\Sigma}, \{ \Sigma_{\varepsilon} \}_{\varepsilon \in \mathrm{U}(1)} )$
is a reflection of $\mathfrak{a}$ with respect to the hyperplane $\{H \in \mathfrak{a} \mid \langle \alpha , H\rangle = n \pi -\varphi_{\varepsilon}\}$.
Thus, we can take a fundamental domain $P_{0}$ of the action of $W(\tilde{\Sigma}, \{ \Sigma_{\varepsilon} \}_{\varepsilon \in \mathrm{U}(1)} )$ 
on $\mathfrak{a}$ as follows;
\begin{align*}
P_{0} =\bigcap_{ \varepsilon \in \mathrm{U}(1)} P_{\varepsilon} 
\end{align*}
where
\begin{align*}
P_{\varepsilon}=
\begin{cases}
\{ H \in \mathfrak{a} \mid -\varphi_{\varepsilon} < \langle \alpha , H\rangle < \pi-\varphi_{\varepsilon}\ (\alpha \in \Sigma_{\varepsilon}^{+})\} &  (\varphi_{\varepsilon} \geq 0)\\
\{ H \in \mathfrak{a} \mid -\pi -\varphi_{\varepsilon} < \langle \alpha , H\rangle < \varphi_{\varepsilon}\ (\alpha \in \Sigma_{\varepsilon}^{+})\} & (\varphi_{\varepsilon} < 0)
\end{cases} 
\end{align*}
for $\varepsilon \in \mathrm{U}(1)$ satisfying $\Sigma_{\varepsilon}\neq \emptyset $.
Then $P_{0}$ is a nonempty open set in $\mathfrak{a}$, 
and 
is a fundamental domain $P_{0}$ of the action of $(\tilde{\Sigma}, \{ \Sigma_{\varepsilon} \}_{\varepsilon \in \mathrm{U}(1)} )$ 
on $\mathfrak{a}$.
\section{On geometry of orbits of Hermann actions }\label{sect-geom}
In this section, 
we consider geometric properties of orbits of Hermann actions.
In section \ref{sect-2nd}, we compute the second fundamental form of orbits, 
and give characterization of minimal, austere and  totally geodesic orbits.

In section \ref{sect-wr}, 
we give sufficient conditions for orbits to be arid (resp. weakly reflective) submanifolds. 

\subsection{Second fundamental form}\label{sect-2nd}
For $H \in \mathfrak{a}$, set $g=\exp(H),\ p_{1}=\pi_{1}(g)$. 
We consider the orbit $K_{2}p_{1}\subset M_{1}$.
First, we describe the tangent space and normal space of the orbit $K_{2}p_{1}$ using $W(\tilde{\Sigma}, \{ \Sigma_{\varepsilon} \}_{\varepsilon \in \mathrm{U}(1)} )$.

In general, 
\begin{align*}
T_{p_{1}}K_{2}p_{1} 
&=\left\{ \left. \left. \frac{d}{dt} \exp (tX) p_{1} \right|_{t=0} \ \right| \  X \in \mathfrak{k}_{2} \right\} \\
&=\left\{ \left. \left. \frac{d}{dt} \pi_{1}( \exp (tX) g) \right|_{t=0} \ \right| \  X \in \mathfrak{k}_{2} \right\}\\
&=\left\{ \left. \left. \frac{d}{dt} \pi_{1}(g  \exp (t \mathrm{Ad}(g)^{-1} X) ) \right|_{t=0} \ \right| \  X \in \mathfrak{k}_{2} \right\}\\
&=dL_{g} d\pi_{1} (\mathrm{Ad}(g)^{-1} \mathfrak{k}_{2}),\\
T_{p_{1}}^{\perp}K_{2}p_{1}
&=dL_{g}(\mathfrak{m}_{1}\cap \mathrm{Ad}(g)^{-1} \mathfrak{m}_{2}) 
\end{align*}
holds.
 
\begin{lem}\label{Lem-5}
\begin{align*}
d\pi_{1} (\mathfrak{g}(0) \cap \mathfrak{k}_{2})
=(V(0, -1)\cap \mathfrak{m}_{1} )\oplus \sum_{\varepsilon \in \mathrm{U}(1) \atop \mathrm{Im}(\varepsilon)>0} (V(0, \varepsilon) \cap \mathfrak{m}_{1})
\end{align*}
\end{lem}
\begin{proof}
Since 
\begin{align*}
\mathfrak{g}(0)\cap \mathfrak{g}=V(0,1)\oplus V(0, -1)\oplus \sum_{\varepsilon \in \mathrm{U}(1) \atop \mathrm{Im}(\varepsilon) >0} V(0, \varepsilon),
\end{align*}
we have 
\begin{align*}
\mathfrak{g}(0) \cap \mathfrak{k}_{2}
=(V(0, 1)\cap \mathfrak{k}_{2} )\oplus(V(0, -1)\cap \mathfrak{k}_{2} )\oplus \sum_{\varepsilon \in \mathrm{U}(1) \atop \mathrm{Im}(\varepsilon)>0} (V(0, \varepsilon) \cap \mathfrak{k}_{2}).
\end{align*}
Then since 
 $(V(0, 1)\cap \mathfrak{k}_{2} )\subset \mathfrak{k}_{1}\cap \mathfrak{k}_{2}$
 and 
 $(V(0, -1)\cap \mathfrak{k}_{2} )\subset \mathfrak{m}_{1}\cap \mathfrak{k}_{2}$, 
 $d\pi_{1}((V(0, 1)\cap \mathfrak{k}_{2} ))=\{0\}$
 and 
 $d\pi_{1}( (V(0, -1)\cap \mathfrak{k}_{2} ))= V(0, -1)\cap \mathfrak{m}_{1} $
 holds, respectively.
 
 For each 
$\varepsilon \in \mathrm{U}(1)$ with $\mathrm{Im}(\varepsilon)>0$, 
from 
$V(0, \varepsilon)=(\mathfrak{g}(0, \varepsilon)\oplus \mathfrak{g}(0, \varepsilon^{-1}))\cap \mathfrak{g}$
and $\overline{\mathfrak{g}(0, \varepsilon)}=\mathfrak{g}(0, \varepsilon^{-1})$, 
for any 
$X\in V(0, \varepsilon) \cap \mathfrak{k}_{2}$, 
there exists
$Y \in \mathfrak{g}(0, \varepsilon)$ 
such that 
$X=Y +\overline{Y}$.
Then we have 
\begin{align*}
d\pi_{1}(X) 
=&\frac{X- \theta_{1}(X)}{2}
=\frac{X- \theta_{1}\theta_{2}(X)}{2}
=\frac{Y +\overline{Y}- \theta_{1}\theta_{2}(Y +\overline{Y})}{2}\\
=&\frac{(1-\varepsilon)Y +(1-\varepsilon^{-1})\overline{Y}}{2}.
\end{align*}
Hence if $d\pi_{1}(X)=0$, then $X=0$ holds.

Conversely, 
for $X \in V(0, \varepsilon) \cap \mathfrak{m}_{1}$, 
set 
\begin{align*}
X'=\frac{X+\theta_{2}(X)}{1-\mathrm{Re}(\varepsilon)}.
\end{align*}
Then we have 
$X'\in V(0, \varepsilon)$ and $d\pi (X')=X$.

Therefore, 
we have 
$d\pi (V(0, \varepsilon) \cap \mathfrak{k}_{2})=V(0, \varepsilon) \cap \mathfrak{m}_{1}$.
\end{proof}
For each 
$\varepsilon \in \mathrm{U}(1)$ and $\alpha\in \Sigma_{\varepsilon}^{+}$, 
from Lemma~\ref{onb}, 
\begin{align*}
V(\alpha, \varepsilon) \cap \mathfrak{k}_{2}= \sum_{i=1}^{m(\alpha, \varepsilon)} \mathbb{R}\cdot \widetilde{X}_{\alpha, i}^{\varepsilon}.
\end{align*}
Thus we have 
\begin{align*}
\mathrm{Ad}(g)^{-1} (V(\alpha, \varepsilon) \cap \mathfrak{k}_{2})
=&\sum_{i=1}^{m(\alpha, \varepsilon)} \mathbb{R}\cdot (\cos \langle \alpha, H \rangle \widetilde{X}_{\alpha, i}^{\varepsilon}-\sin \langle \alpha, H \rangle \widetilde{Y}_{\alpha, i}^{\varepsilon})\\
=&\sum_{i=1}^{m(\alpha, \varepsilon)} \mathbb{R}\cdot (\cos (\langle \alpha, H \rangle +\varphi_{\varepsilon}) X_{\alpha, i}^{\varepsilon}-\sin (\langle \alpha, H \rangle +\varphi_{\varepsilon}) Y_{\alpha, i}^{\varepsilon}).
\end{align*}
Moreover, from Lemma~\ref{Lem-5} 
we obtain
\begin{align*}
&dL_{g}^{-1} (T_{p_{1}}K_{2}p_{1})\\
=&(V(0, -1)\cap \mathfrak{m}_{1} )\oplus \sum_{\varepsilon \in \mathrm{U}(1) \atop \mathrm{Im}(\varepsilon)>0} (V(0, \varepsilon) \cap \mathfrak{m}_{1})
\oplus \sum_{\varepsilon \in \mathrm{U}(1) } \sum_{\alpha\in \Sigma_{\varepsilon}^{+} \atop \langle \alpha, H \rangle +\varphi_{\varepsilon} \not\in \pi  \mathbb{Z}} (V(\alpha, \varepsilon) \cap \mathfrak{m}_{1})
\end{align*}
and 
\begin{align*}
\mathfrak{m}_{1}
=&(\mathfrak{g}(0)\cap \mathfrak{m}_{1})\oplus \sum_{\varepsilon \in \mathrm{U}(1) } \sum_{\alpha\in \Sigma_{\varepsilon}^{+}} (V(\alpha, \varepsilon) \cap \mathfrak{m}_{1})\\
=&\mathfrak{a}\oplus (V(0, -1)\cap \mathfrak{m}_{1} )\oplus \sum_{\varepsilon \in \mathrm{U}(1) \atop \mathrm{Im}(\varepsilon)>0} (V(0, \varepsilon) \cap \mathfrak{m}_{1})
\oplus \sum_{\varepsilon \in \mathrm{U}(1) } \sum_{\alpha\in \Sigma_{\varepsilon}^{+}} (V(\alpha, \varepsilon) \cap \mathfrak{m}_{1}).
\end{align*}
Therfoer
\begin{align*}
dL_{g}^{-1} (T^{\perp}_{p_{1}}K_{2}p_{1})
=\mathfrak{a} \oplus \sum_{\varepsilon \in \mathrm{U}(1) } \sum_{\alpha\in \Sigma_{\varepsilon}^{+} \atop \langle \alpha, H \rangle +\varphi_{\varepsilon} \in \pi \mathbb{Z} } (V(\alpha, \varepsilon) \cap \mathfrak{m}_{1}).
\end{align*}

Next, we compute the second fundamental form of the orbit $K_{2}p_{1}\subset M_{1}$.

For each $X\in \mathfrak{g}$, 
We define a vector field $X^{\ast}$ on $M_{1}$
as follows:
\begin{align*}
(X^{\ast})_{p}=\left. \frac{d}{dt} \exp (tX) p\right|_{t=0} \quad (p\in M_{1}).
\end{align*}

The following lemma is known.
\begin{lem}[\cite{I1}, Lemma~4.13, (p.113)] \label{lem-nabla}
Denote by $\widetilde{\nabla}$ the Levi-Civita connection on $M_{1}$.
\begin{enumerate}
\item For $g' \in G$, and $X, Y\in \mathfrak{g}$, 
\begin{align*}
dL_{g'}(\widetilde{\nabla}_{X^{\ast}}Y^{\ast})= \widetilde{\nabla}_{(\mathrm{Ad}(g') X)^{\ast}}(\mathrm{Ad}(g') Y)^{\ast},
\end{align*}
\item For $X, Y \in \mathfrak{g}$, 
\begin{align*}
(\widetilde{\nabla}_{X^{\ast}} Y^{\ast})_{\pi_{1}(e)} =
\begin{cases}
-[X, Y] & (X\in \mathfrak{m}_{1})\\
0& (X\in \mathfrak{k}_{1}),
\end{cases}
\end{align*}
\item 
For $p=\pi_{1}(g')\in M_{1}$,
\begin{align*}
(\widetilde{\nabla}_{X^{\ast}} Y^{\ast})_{p} =-dL_{g'}\big( [(\mathrm{Ad}(g')^{-1}X) _{\mathfrak{m}_{1}}, \mathrm{Ad}(g')^{-1}Y]\big)_{\mathfrak{m}_{1}}
\end{align*}
\end{enumerate}
\end{lem}
For $p_{1}=\pi_{1}(\exp (H)) \in M_{1}$ and $X\in \mathfrak{g}$, 
since 
\begin{align*}
(X^{\ast})_{p_{1}} =\left. \frac{d}{dt} \exp(tX)p_{1} \right|_{t=0} 
=dL_{g}\circ d\pi_{1}(\mathrm{Ad}(\exp H)^{-1}X)
\end{align*}
\begin{itemize}
\item For $X\in V(0, -1)\cap \mathfrak{m}_{1}$,  
$(X^{\ast})_{p_{1}} =dL_{g}X$.
\item For $\varepsilon \in \mathrm{U}(1)$ with $\mathrm{Im}(\varepsilon)>0$ and $X\in V(0, \varepsilon)\cap \mathfrak{m}_{1}$, 
\begin{align*}
\left(\frac{X+\theta_{2}(X)}{1-\mathrm{Re}(\varepsilon)} \right)^{\ast}_{p_{1}} =dL_{g}(X).
\end{align*}
\item For $\varepsilon \in \mathrm{U}(1)$, $\alpha \in \Sigma_{\varepsilon}^{+}$ ( $\langle \alpha, H\rangle+\varphi_{\varepsilon} \not\in \pi \mathbb{Z} $) and 
$1\leq i\leq m(\alpha, \varepsilon)$, 
\begin{align*}
(\widetilde{X}_{\alpha, i}^{\varepsilon})^{\ast}_{p_{1}} 
&=dL_{g}\circ d\pi_{1}(\mathrm{Ad}(g)^{-1} \widetilde{X}_{\alpha, i}^{\varepsilon})\\
&=dL_{g}\circ d\pi_{1}(\cos (\langle \alpha, H \rangle +\varphi_{\varepsilon})X_{\alpha, i}^{\varepsilon} -\sin (\langle \alpha, H \rangle +\varphi_{\varepsilon})Y_{\alpha, i}^{\varepsilon})\\
&=dL_{g}( -\sin (\langle \alpha, H \rangle +\varphi_{\varepsilon})Y_{\alpha, i}^{\varepsilon})
\end{align*}
holds. 
Thus we have 
\begin{align*}
- \left( \frac{\widetilde{X}_{\alpha, i}^{\varepsilon}}{\sin (\langle \alpha, H \rangle +\varphi_{\varepsilon})}\right)^{\ast}_{p_{1}}=dL_{g}(Y_{\alpha, i}^{\varepsilon}).
\end{align*}
\end{itemize}
For $X \in \mathfrak{k}_{2}$, 
$X^{\ast}$ gives a vector field on the orbit $K_{2}p_{1}$.
Let $B_{H}$ denotes the second fundamental form of $K_{1}p_{1}\subset M_{1}$.
Then we have the following theorem. 
\begin{thm}\label{thm-2ndff}
\begin{enumerate}
\item For $\varepsilon, \delta \in \{a\in \mathrm{U}(1)\setminus \{1\} \mid \mathrm{Im} (a)\geq 0 \}$, $X\in V(0, \varepsilon)\cap \mathfrak{m}_{1}, Y\in V(0, \delta)\cap \mathfrak{m}_{1}$, 
\begin{align*}
dL_{g}^{-1}B_{H}(dL_{g} X, dL_{g} Y)=0
\end{align*}
holds.
\item For $\varepsilon \in \mathrm{U}(1), \alpha \in \Sigma_{\varepsilon}^{+}\ (\langle \alpha, H \rangle +\varphi_{\varepsilon} \not\in \pi \mathbb{Z}), 1\leq i\leq m(\alpha, \varepsilon)$, 
$\delta \in \{a\in \mathrm{U}(1)\setminus \{1\} \mid \mathrm{Im} (a)\geq 0 \}$, $X\in V(0, \delta)\cap \mathfrak{m}_{1}$,
\begin{align*}
dL_{g}^{-1}B_{H}(dL_{g} X, dL_{g} Y_{\alpha, i}^{\varepsilon})=\cot (\langle \alpha, H \rangle +\varphi_{\varepsilon})[X, X_{\alpha, i}^{\varepsilon}]^{\perp}
\end{align*}
holds.
\item For $\varepsilon, \delta \in \mathrm{U}(1), \alpha \in \Sigma_{\varepsilon}^{+}\ (\langle \alpha, H \rangle +\varphi_{\varepsilon} \not\in \pi \mathbb{Z}), 1\leq i\leq m(\alpha, \varepsilon)$, 
$\beta \in \Sigma_{\varepsilon}^{+}\ (\langle \beta, H \rangle +\varphi_{\delta} \not\in \pi \mathbb{Z}), 1\leq j\leq m(\beta, \delta)$, 
\begin{align*}
dL_{g}^{-1}B_{H}( dL_{g} Y_{\alpha, i}^{\varepsilon},  dL_{g} Y_{\beta, j}^{\delta})=\cot (\langle \beta, H \rangle +\varphi_{\delta})[Y_{\alpha, i}^{\varepsilon}, X_{\beta, j}^{\delta}]^{\perp}
\end{align*}
holds.
\end{enumerate}
Here, $X^{\perp}$ denotes the normal component of $X\in \mathfrak{m}_{1}$.
\end{thm}
\begin{proof}
For $\varepsilon, \delta \in \{a\in \mathrm{U}(1)\setminus \{1\} \mid \mathrm{Im} (a)\geq 0 \}$, 
$X\in V(0, \varepsilon)\cap \mathfrak{m}_{1}, Y\in V(0, \delta)\cap \mathfrak{m}_{1}$, 
from 
\begin{align*}
\left(\frac{X+\theta_{2}(X)}{1-\mathrm{Re}(\varepsilon)} \right)^{\ast}_{p_{1}} =dL_{g}(X), 
\left(\frac{Y+\theta_{2}(Y)}{1-\mathrm{Re}(\delta)} \right)^{\ast}_{p_{1}} =dL_{g}(Y)
\end{align*}
and Lemma~\ref{lem-nabla}, we have 
\begin{align*}
dL_{g}^{-1}B_{H}(dL_{g} X, dL_{g} Y)
=-\left[ X,  \frac{Y+\theta_{2} (Y)}{1-\mathrm{Re}(\delta)}\right]^{\perp}_{\mathfrak{m}_{1}}
=-\left[ X,  \left( \frac{Y+\theta_{2} (Y)}{1-\mathrm{Re}(\delta)}\right)_{\mathfrak{k}_{1}} \right]^{\perp} \in \mathfrak{a}.
\end{align*}
For each $H'\in \mathfrak{a}$, we have 
\begin{align*}
\left\langle \left[ X,  \left( \frac{Y+\theta_{2} (Y)}{1-\mathrm{Re}(\delta)}\right)_{\mathfrak{k}_{1}} \right]^{\perp}, H' \right\rangle
&=\left\langle \left[ X,  \left( \frac{Y+\theta_{2} (Y)}{1-\mathrm{Re}(\delta)}\right)_{\mathfrak{k}_{1}} \right], H' \right\rangle\\
&=-\left\langle   \left( \frac{Y+\theta_{2} (Y)}{1-\mathrm{Re}(\delta)}\right)_{\mathfrak{k}_{1}} , [ X, H' ] \right\rangle
=0.
\end{align*}
Thus 
\begin{align*}
\left[ X,  \left( \frac{Y+\theta_{2} (Y)}{1-\mathrm{Re}(\delta)}\right)_{\mathfrak{k}_{1}} \right]^{\perp}=0.
\end{align*}
holds. 

For 2,
since 
\begin{align*}
dL_{g}(Y_{\alpha, i}^{\varepsilon})=- \left( \frac{\widetilde{X}_{\alpha, i}^{\varepsilon}}{\sin (\langle \alpha, H \rangle +\varphi_{\varepsilon})}\right)^{\ast}_{p_{1}}
\end{align*}
and Lemma~\ref{lem-nabla}, we have 
\begin{align*}
dL_{g}^{-1}B_{H}(dL_{g} X, dL_{g} Y_{\alpha, i}^{\varepsilon})
&=-[X, -\cot (\langle \alpha, H \rangle +\varphi_{\varepsilon}) X_{\alpha, i}^{\varepsilon}+ Y_{\alpha, i}^{\varepsilon}]^{\perp}_{\mathfrak{m}_{1}}\\
&=[X, \cot (\langle \alpha, H \rangle +\varphi_{\varepsilon}) X_{\alpha, i}^{\varepsilon}]^{\perp}.
\end{align*}
for each $X\in V(0, \delta)\cap \mathfrak{m}_{1}$.

For 3,
\begin{align*}
dL_{g}^{-1}B_{H}(dL_{g}(Y_{\alpha, i}^{\varepsilon}),dL_{g}(Y_{\beta, j}^{\delta}))
&=-[Y_{\alpha, i}^{\varepsilon}, -\cot (\langle \beta, H \rangle \varphi_{\delta}) X_{\beta, j}^{\delta} +Y_{\beta, j}^{\delta}]^{\perp}_{\mathfrak{m}_{1}}\\
&=[Y_{\alpha, i}^{\varepsilon}, \cot (\langle \beta, H \rangle \varphi_{\delta}) X_{\beta, j}^{\delta}]^{\perp}
\end{align*}
holds.
\end{proof}
From Theorem~\ref{thm-2ndff}, we have the following corollary. 
\begin{cor}
Let $m_{H}$ denotes the mean curvature vector field of the orbit $K_{2}p_{1}\subset M_{1}$.
Then,
\begin{align*}
(m_{H})_{p_{1}}=-dL_{g}\left( \sum_{\varepsilon \in \mathrm{U}(1)} \sum_{\alpha \in \Sigma_{\varepsilon}^{+} \atop \langle \alpha, H \rangle +\varphi_{\varepsilon} \not\in \pi \mathbb{Z}}
m(\alpha, \varepsilon) \cot (\langle \alpha , H\rangle +\varphi_{\varepsilon} ) \alpha \right)
\end{align*}
holds. 
\end{cor}
\begin{proof}
For each $\varepsilon \in \mathrm{U}(1), \alpha \in \Sigma_{\varepsilon}^{+}, 1\leq i\leq m(\alpha, \varepsilon) $, 
from
Theorem~\ref{thm-2ndff},
we have 
\begin{align*}
dL_{g}^{-1}B_{H}( dL_{g} Y_{\alpha, i}^{\varepsilon},  dL_{g} Y_{\alpha, i}^{\varepsilon})
=\cot (\langle \alpha, H \rangle +\varphi_{\varepsilon})[Y_{\alpha, i}^{\varepsilon}, X_{\alpha, i}^{\varepsilon}]^{\perp}
=-\cot (\langle \alpha, H \rangle +\varphi_{\varepsilon})\alpha.
\end{align*}
\end{proof}

Next, we consider austereness of orbits.
We review the definition of the austere submanifolds.
\begin{defn}[\cite{HL}]
Let $M$ be a submanifold of Riemannian manifold $\tilde{M}$. 
We denote the shape operator of $M$ by $A$. 
Then $M$ is called an austere submanifold if for each normal vector $\xi \in T^{\perp}_{x}$, 
the set of eigenvalues with their multiplicities of $A^{\xi}$ is invariant under the multiplication by $-1$.
\end{defn}
The notion of austere submanifolds are introduced in \cite{HL}.
By definition, 
we can see that austere submanifolds are minimal submanifolds.

Hereafter, 
we denote the shape operator of $K_{2} p_{1}\subset M_{1}$ with respect to $dL_{g}(\xi) \in T^{\perp}_{p_{1}} K_{2}p_{1}$ with the symbol $A^{\xi}$.

In order to investigate the eigenvalues of the shape operator $A^{\xi}$, 
we can assume $\xi \in \mathfrak{a}$ without loss of generality.

Because, 
the isotropy subgroup $(K_{2})_{p_{1}}$ of $K_{2}$ at $p_{1}$ is expressed as 
\begin{align*}
(K_{2})_{p_{1}}
=&\{k \in K_{2} \mid kp_{1}=p_{1}\}
=\{k \in K_{2} \mid kg\pi_{1} (e)=g\pi_{1}(e) \}\\
=&\{k \in K_{2} \mid g^{-1} k g \in K_{1} \}
=K_{2}\cap (gK_{1}g^{-1}).
\end{align*}
Since elements in $(K_{2})_{p_{1}}$ fix $p_{1}$, 
$(K_{2})_{p_{1}}$ has representation on $T^{\perp}_{p_{1}}K_{2}p_{1}$ by differential of the action.
Actually, 
for $k\in (K_{2})_{p_{1}}$ and $dL_{g}\xi \in T^{\perp}_{p_{1}}K_{2}p_{1}$, 
\begin{align*}
dL_{k}(dL_{g}\xi )
=&\left. \frac{d}{dt} k \mathrm{Exp}_{p_{1}}(t dL_{g}\xi) \right|_{t=0}
=\left. \frac{d}{dt} k g\exp (t\xi) K_{1} \right|_{t=0}\\
=&\left. \frac{d}{dt} g g^{-1}k g\exp (t\xi) g^{-1}k^{-1}g K_{1} \right|_{t=0}
=dL_{g}(\mathrm{Ad}(g^{-1}k g) \xi)
\end{align*} 
holds.
Thus, 
the representation of $(K_{2})_{p_{1}}$ on $T^{\perp}_{p_{1}}K_{2}p_{1}$
is equivalent to the adjoint representation of $g^{-1}(K_{2})_{p_{1}}g$ on $\mathfrak{m}_{1}\cap \mathrm{Ad}(g)^{-1}\mathfrak{m}_{2}$.
Moreover, 
the Lie algebra of $g^{-1}(K_{2})_{p_{1}}g$ is $\mathfrak{k}_{1}\cap \mathrm{Ad}(g)^{-1}\mathfrak{k}_{2}$, 
and 
$((\mathfrak{k}_{1}\cap \mathrm{Ad}(g)^{-1}\mathfrak{k}_{2})\oplus (\mathfrak{m}_{1}\cap \mathrm{Ad}(g)^{-1}\mathfrak{m}_{2}) , \theta_{1})$
is a symmetric orthogonal Lie algebra, 
we have 
\begin{align}\label{eq-slicerep}
\mathfrak{m}_{1}\cap \mathrm{Ad}(g)^{-1}\mathfrak{m}_{2}=\mathrm{Ad}(g^{-1}(K_{2})_{p_{1}}g) \mathfrak{a}.
\end{align}
Therefore, 
for all $dL_{g} \xi \in T^{\perp}_{p_{1}}K_{2}p_{1}$, there exist 
$\xi' \in \mathfrak{a}$ and $k\in K_{1}\cap (g^{-1}K_{2}g)$, 
such that 
\begin{align} \label{eq-slicerep2}
dL_{g}\xi =dL_{g}\mathrm{Ad}(k)\xi'.
\end{align}

From this fact and the relation between shape operator and second fundamental form, 
we can compute as follows;
\begin{align*}
\langle A^{\xi } X, Y\rangle
=&\langle B_{H}( X, Y), dL_{g}\xi\rangle
= \langle B_{H}( X, Y), dL_{g}\mathrm{Ad}(k)\xi' \rangle \\
=& \langle B_{H}( X, Y), dL_{g}dL_{k}\xi' \rangle
= \langle dL_{g}dL_{k^{-1}}dL_{g^{-1}}B_{H}( X, Y), dL_{g}\xi' \rangle \\
=& \langle dL_{gk^{-1}g^{-1}}B_{H}( X, Y), dL_{g}\xi' \rangle \\
= &\langle B_{H}(dL_{gk^{-1}g^{-1}}( X),dL_{gk^{-1}g^{-1}}( Y) ), dL_{g}\xi' \rangle\\
= &\langle A^{\xi'} dL_{gk^{-1}g^{-1}}( X), dL_{gk^{-1}g^{-1}}( Y)  \rangle
\end{align*}
for all $X, Y \in T_{p_{1}}K_{2}p_{1}$.
Hence, we have 
\begin{align*}
A^{\xi }=dL_{gkg^{-1}} A^{\xi'} dL_{gk^{-1}g^{-1}}.
\end{align*}
From the above equation, 
the eigenvalues with their multiplicity of $A^{\xi}$ and $A^{\xi'}$ coincides.

In the following, 
we investigate the eigenvalues of $A^{\xi}$ for each $\xi \in \mathfrak{a}$.
\begin{itemize}
\item 
For each $\varepsilon \in \mathrm{U}(1)$ and $X, Y \in V(0, \varepsilon) \cap \mathfrak{m}_{1}$, 
we have 
\begin{align*}
\langle A^{\xi} dL_{g}X, dL_{g}Y\rangle
=\langle B_{H}(dL_{g}X, dL_{g}Y), dL_{g}\xi \rangle
=0. 
\end{align*}
\item 
For $\varepsilon \in \mathrm{U}(1)$, $X\in V(0, \varepsilon) \cap \mathfrak{m}_{1}$ and 
$\delta \in \mathrm{U}(1), \alpha \in \Sigma_{\delta}^{+} \ (\langle \alpha, H\rangle +\varphi_{\delta} \not\in \pi \mathbb{Z}), 1\leq i\leq m(\alpha, \delta)$,
\begin{align*}
&\langle A^{\xi} dL_{g}X, dL_{g}Y_{\alpha, i}^{\delta} \rangle
=\langle  B_{H}(dL_{g}X, dL_{g}Y_{\alpha, i}^{\delta}) , dL_{g}(\xi) \rangle\\
=&\langle  \cot (\langle \alpha, H\rangle +\varphi_{\delta})dL_{g}[X, X_{\alpha, i}^{\delta}]^{\perp} , dL_{g}(\xi) \rangle\\
=&\langle  \cot (\langle \alpha, H\rangle +\varphi_{\delta})dL_{g}[X, X_{\alpha, i}^{\delta}] , dL_{g}(\xi) \rangle\\
=&\langle  \cot (\langle \alpha, H\rangle +\varphi_{\delta})[X, X_{\alpha, i}^{\delta}] , \xi \rangle\\
=&-\langle  \cot (\langle \alpha, H\rangle +\varphi_{\delta}) X_{\alpha, i}^{\delta} ,[X, \xi] \rangle
=0
\end{align*} 
holds.
\item 
For $\varepsilon \in \mathrm{U}(1), \alpha \in \Sigma_{\varepsilon}^{+} \ (\langle \alpha, H\rangle +\varphi_{\varepsilon} \not\in \pi \mathbb{Z}), 1\leq i\leq m(\alpha, \varepsilon)$ and
$\delta \in \mathrm{U}(1), \beta \in \Sigma_{\delta}^{+} \ (\langle \beta, H\rangle +\varphi_{\delta} \not\in \pi \mathbb{Z}), 1\leq j\leq m(\beta, \delta)$,
\begin{align*}
&\langle A^{\xi} dL_{g}Y_{\alpha, i}^{\varepsilon }, dL_{g}Y_{\beta, j}^{\delta} \rangle
=\langle  B_{H}(dL_{g}Y_{\alpha, i}^{\varepsilon }, dL_{g}Y_{\beta, j}^{\delta} ), dL_{g}(\xi)\rangle\\
=&\langle  \cot (\langle \beta, H\rangle +\varphi_{\delta})dL_{g} [Y_{\alpha, i}^{\varepsilon }, X_{\beta, j}^{\delta} ]^{\perp}, dL_{g}(\xi)\rangle\\
=&\langle  \cot (\langle \beta, H\rangle +\varphi_{\delta}) [Y_{\alpha, i}^{\varepsilon }, X_{\beta, j}^{\delta} ], \xi\rangle\\
=&\cot (\langle \beta, H\rangle +\varphi_{\delta})\langle   Y_{\alpha, i}^{\varepsilon }, [X_{\beta, j}^{\delta} , \xi ]\rangle\\
=&\cot (\langle \beta, H\rangle +\varphi_{\delta})\langle   Y_{\alpha, i}^{\varepsilon },  - \langle \beta,  \xi \rangle Y_{\beta, j}^{\delta} \rangle\\
=&- \langle \beta,  \xi \rangle\cot (\langle \beta, H\rangle +\varphi_{\delta})\langle   Y_{\alpha, i}^{\varepsilon },  Y_{\beta, j}^{\delta} \rangle
\end{align*}
holds.
\end{itemize}
From the above computation, 
for each $\xi \in \mathfrak{a}$,  
$A^{\xi} dL_{g}X=0 $ for 
$\varepsilon \in \mathrm{U}(1), X\in V(0, \varepsilon) \cap \mathfrak{m}_{1}$
and
$A^{\xi} dL_{g}Y_{\alpha, i}^{\varepsilon} =- \langle \alpha,  \xi \rangle\cot (\langle \alpha, H\rangle +\varphi_{\varepsilon})Y_{\alpha, i}^{\varepsilon}$
for  
$\varepsilon \in \mathrm{U}(1), \alpha \in \Sigma_{\varepsilon}^{+} \ (\langle \alpha, H\rangle +\varphi_{\varepsilon} \not\in \pi \mathbb{Z}), 1\leq i\leq m(\alpha, \varepsilon)$
holds. 

Therefore, 
The set of eigenvalues of $A^{\xi}$ is 
\begin{align*}
&\{0 \quad (\text{multiplicity}=m )\}\\
\cup&
\{ - \langle \alpha,  \xi \rangle\cot (\langle \alpha, H\rangle +\varphi_{\varepsilon})  \quad (\text{multiplicity}=m(\alpha, \varepsilon) ) \\
& \qquad \qquad\mid  \varepsilon \in \mathrm{U}(1), \alpha \in \Sigma_{\varepsilon}^{+} \ (\langle \alpha, H\rangle +\varphi_{\varepsilon} \not\in \pi \mathbb{Z})\},
\end{align*}
where $m=\dim \left( (V(0, -1)\cap \mathfrak{m}_{1} )\oplus \sum_{\varepsilon \in \mathrm{U}(1) \atop \mathrm{Im}(\varepsilon)>0} (V(0, \varepsilon) \cap \mathfrak{m}_{1})\right) $.

It is known that the following lemma.
\begin{lem}[\cite{IST2}, p.459]\label{lem-ist}
Let $A$ be a finite subset of metric vector space $(\mathfrak{a}, \langle , \rangle )$.
Then (i) and (ii) are equivalent.
\begin{enumerate}
\item[(i)] For any $\xi \in \mathfrak{a}$, the set $\{ \langle a , \xi \rangle \mid a\in A\}$ with multiplicity is invariant under the multiplication by $-1$.
\item[(ii)] The set $A$ is invariant under the multiplication by $-1$.
\end{enumerate} 
\end{lem}
From Lemma~\ref{lem-ist} we have the following proposition, 
\begin{prop} \label{prop-austere}
The orbit $K_{2}p_{1} \subset M_{1}$ is austere if and only if 
the set 
\begin{align*}
E_{H}
&=\{ - \cot (\langle \alpha, H\rangle +\varphi_{\varepsilon}) \alpha  \quad (\text{multiplicity}=m(\alpha, \varepsilon) ) \\
& \qquad \qquad\mid  \varepsilon \in \mathrm{U}(1), \alpha \in \Sigma_{\varepsilon}^{+} \ (\langle \alpha, H\rangle +\varphi_{\varepsilon} \not\in \pi \mathbb{Z})\}
\end{align*}
is invariant under the multiplication by $-1$.
\end{prop}

Next, 
we consider a condition for the orbit $K_{2}p_{1}\subset M_{1}$ to be totally geodesic submanifold.
A necessary and sufficient condition for a Hermann action has totally geodesic orbits is given by an algebraic condition of  $\theta_{1}$ and $ \theta_{2}$ (cf. Theorem~\ref{thm-iff-tot-geod}).
To explain this fact, we define the notion of irreducibility of a triple $(\mathfrak{g}, \theta_{1}, \theta_{2})$.
\begin{defn}
A triple $(\mathfrak{g}, \theta_{1}, \theta_{2})$ is called irreducible,
if there is no non-torivial ideal $\mathfrak{g}_{0}$ which satisfies 
$\theta_{1}(\mathfrak{g}_{0})=\theta_{2}(\mathfrak{g}_{0})=\mathfrak{g}_{0}$.
\end{defn}
In \cite{M2}, 
irreducible $(\mathfrak{g}, \theta_{1}, \theta_{2})$ are classified. 
When  $\mathfrak{g}$ is semisimple, 
we can decomposed as 
\begin{align*}
\mathfrak{g}=\mathfrak{f}_{1}\oplus \cdots \oplus \mathfrak{f}_{l},\
\theta_{1}=\sigma_{1}\oplus \cdots \oplus \sigma_{l},\ 
\theta_{2}=\tau_{1}\oplus \cdots \oplus \tau_{l}
\end{align*}
where 
$(\mathfrak{f}_{i}, \sigma_{i}, \tau_{i})\ (i=1, \ldots, l)$ are irreducible symmetric triads.

From Theorem~\ref{thm-2ndff}, 
we can see the following lemma.
\begin{lem}\label{lem-iff-tot-geod}
The orbit $K_{2}p_{1} \subset M_{1}$ is totally geodesic if and only if 
$\langle \alpha, H \rangle +\varphi_{\varepsilon} \in \frac{\pi}{2} \mathbb{Z}$ 
holds for all $\varepsilon \in \mathrm{U}(1)$ and $\alpha \in \Sigma_{\varepsilon}^{+}$.
\end{lem}
By using Lemma~\ref{lem-iff-tot-geod}, 
we can prove the following theorem. 
\begin{thm}\label{thm-iff-tot-geod}
Let 
$(\mathfrak{g}, \theta_{1}, \theta_{2})$ be irreducible.
We suppose $\mathfrak{a}\neq \{0\}$. 
Then 
the Hermann action $K_{2}\curvearrowright M_{1}$ has a totally geodesic orbit 
if and only if 
$(\theta_{1}\mathrm{Ad}(g) \theta_{2}\mathrm{Ad}(g)^{-1})^{2}=\mathrm{id}$
holds for some element $g\in G$.
\end{thm}
\begin{proof}
First, we suppose 
the Hermann action $K_{2}\curvearrowright M_{1}$ has a totally geodesic orbit.
Then, since 
$G=K_{2}\exp(\mathfrak{a}) K_{1}$, 
there exists $H \in \mathfrak{a}$, such that 
$K_{2}\pi_{1}(\exp(H)) \subset M_{1}$ is totally geodesic. 
From  Lemma~\ref{lem-iff-tot-geod}, 
for each $\varepsilon \in \mathrm{U}(1), \alpha \in \Sigma_{\varepsilon}^{+}$,  
$\langle \alpha, H \rangle +\varphi_{\varepsilon} \in \frac{\pi}{2}\mathbb{Z}$ holds.

Moreover, for $\varepsilon, \delta \in \mathrm{U}(1)$ satisfying $\mathfrak{g}(\alpha, \varepsilon)\neq \{0\}$ and 
$\mathfrak{g}(\alpha, \delta )\neq \{0\}$, 
we have
\begin{align*}
&\langle \alpha, H \rangle +\varphi_{\varepsilon} \in \frac{\pi}{2}\mathbb{Z},  \quad \langle \alpha, H \rangle +\varphi_{\delta} \in \frac{\pi}{2}\mathbb{Z}
\end{align*}
Thus we have 
\begin{align*}
\varphi_{\varepsilon} - \varphi_{\delta} \in \frac{\pi}{2}\mathbb{Z}.
\end{align*}
Since $\varphi_{\varepsilon}, \varphi_{\delta}\in (-\pi /2 , \pi/2]$, 
we get 
$\varphi_{\varepsilon} - \varphi_{\delta}= -\pi /2,\ 0,\ \pi /2, $ and 
$\varepsilon =e^{2\sqrt{-1}\varphi_{\varepsilon}}, \delta =e^{2\sqrt{-1}\varphi_{\delta}}$.
Thus 
$\varepsilon \delta^{-1} =\pm1$ holds.
This means that 
$\varepsilon=\pm \delta$. 
Therefore, 
for each 
$\varepsilon \in \mathrm{U}(1), \alpha \in \Sigma_{\varepsilon}$, 
we have $\mathfrak{g}(\alpha)=\mathfrak{g}(\alpha, \varepsilon)\oplus \mathfrak{g}(\alpha, -\varepsilon)$.
Thus for all 
$X\in \mathfrak{g}(\alpha)$, 
\begin{align*}
(\theta_{1}\theta_{2})^2 X=\varepsilon^{2}X.
\end{align*}
holds.

On the other hand, 
for all $X \in \mathfrak{g}(\alpha, \pm \varepsilon)$, we have  
\begin{align*}
\mathrm{Ad}(\exp H)^{4}X= \exp (\pm \sqrt{-1}\langle \alpha, 4H\rangle ) X.
\end{align*}
Since 
$\langle \alpha, H \rangle +\varphi_{\varepsilon} \in \frac{\pi}{2}\mathbb{Z}$, 
we have 
\begin{align*}
\mathrm{Ad}(\exp H)^{-4}X= \varepsilon^{2}X.
\end{align*}
Thus, 
we have 
\begin{align*}
(\theta_{1}\theta_{2})^2 \mathrm{Ad}(\exp H)^{4}X=(\theta_{1} \mathrm{Ad}(\exp H)^{-1} \theta_{2} \mathrm{Ad}(\exp H) )^{2}X=X.
\end{align*}

By the above arguments, 
for each $\varepsilon \in \mathrm{U}(1), \alpha \in \Sigma_{\varepsilon}^{+}$ and $X \in \mathfrak{g}(\alpha, \pm \varepsilon)$, 
$(\theta_{1} \mathrm{Ad}(\exp H)^{-1} \theta_{2} \mathrm{Ad}(\exp H) )^{2} (X)=X$
holds.

Denote by $\mathfrak{f}$ the subalgebra generated by $\sum_{\alpha \in \tilde{\Sigma} }\mathfrak{g}(\alpha)$.
Then $\mathfrak{f}$ is an ideal of $\mathfrak{g}^{\mathbb{C}}$ and $\theta_{i}(\mathfrak{f})=\mathfrak{f}\ (i=1,2 )$ holds. 
Since $(\mathfrak{g}, \theta_{1}, \theta_{2})$ is irreducible, 
$\mathfrak{f}=\mathfrak{g}^{\mathbb{C}}$.
This means that 
the automorphism $(\theta_{1} \mathrm{Ad}(\exp H)^{-1} \theta_{2} \mathrm{Ad}(\exp H) )^{2}$ is the identity map on $\mathfrak{g}^{\mathbb{C}}$.
When we set $g=\exp (-H)$, we have 
$(\theta_{1}\mathrm{Ad}(g) \theta_{2}\mathrm{Ad}(g)^{-1})^{2}=\mathrm{id}$.

Next, we suppose 
$(\theta_{1}\mathrm{Ad}(g) \theta_{2}\mathrm{Ad}(g)^{-1})^{2}=\mathrm{id}$
holds for some element $g\in G$.
Then $(\theta_{1}\mathrm{I}_{g} \theta_{2}\mathrm{I}_{g}^{-1})^{2}=\mathrm{id}$ holds on $G$ 
where $\mathrm{I}_{g}(h) =ghg^{-1} \ (h \in G)$.
We set $\theta_{2}'=\mathrm{I}_{g} \theta_{2}\mathrm{I}_{g}^{-1}$.
Then $\theta_{2}'$ is an involutive automorphism on $G$ which satisfies $\theta_{2}'(K_{1}) =K_{1}$.
Hence, $\theta_{2}'$ induces an involutive isometry $\sigma$ on $M_{1}$.
For each $g\pi_{1}(k_{2}g^{-1}) \in L_{g}(K_{2}\pi_{1}(g^{-1}))$, 
we can see that
\begin{align*}
\sigma (g\pi_{1}(k_{2}g^{-1}))
&=\pi_{1}(\mathrm{I}_{g} \theta_{2}\mathrm{I}_{g}^{-1}(gk_{2}g^{-1}))
=\pi_{1}(\mathrm{I}_{g} \theta_{2}(k_{2}))\\
&=\pi_{1}(\mathrm{I}_{g} (k_{2}))
=\pi_{1}(gk_{2} g^{-1})
=g\pi_{1}(k_{2}g^{-1}).
\end{align*}
Moreover, by definition of $\sigma$,  
we have $(d\sigma)_{\pi_{1}(e)} (\xi )=-\xi$
for each normal vector $\xi \in T^{\perp}_{\pi_{1}(e) } L_{g}(K_{2}\pi_{1}(g^{-1})) =\mathfrak{m}_{1}\cap \mathrm{Ad}(g) \mathfrak{m}_{2}$.
Therefore, $L_{g}(K_{2}\pi_{1}(g^{-1})$
is a reflective submanifold of $M_{1}$.
Since $L_{g}$ is an isometry, 
the orbit $K_{2}\pi_{1}(g)$ of the action 
$K_{2}\curvearrowright M_{1}$ is a reflective submanifold.
Therefore, $K_{2}\pi_{1}(g)$ is totally geodesic submanifold of $M_{1}$.
\end{proof}

\subsection{Weakly reflective submanifolds and arid submanifolds}\label{sect-wr}
The notion of weakly reflective submanifold which is introduced in \cite{IST2} is a generalization of the notion of reflective submanifold.
The notion of arid submanifold is introduced in \cite{Taketomi} as a generalization of the notion of  weakly reflective submanifold.
Arid submanifolds and weakly reflective submanifold are minimal submanifolds and austere submanifolds, respectively.

First, we recall the definitions of reflective submanifolds, weakly refrective submanifolds and arid submanifolds.
 \begin{defn}[\cite{L}]
 Let $M$ be a submanifold of Riemannian manifold $(\tilde{M}, \langle , \rangle)$.
 Then $M$ is a reflective submanifold of $\tilde{M}$ if there exists an involutive isometry $\sigma_{M}$ of $\tilde{M}$
 such that $M$ is a connected component of the fixed point set of $\sigma_{M}$.
 \end{defn}
 \begin{defn} [\cite{IST2}]
 Let $M$ be a submanifold of a Riemannian manifold $(\tilde{M}, \langle , \rangle)$.
 For each normal vector $\xi \in T^{\perp}_{x}M$ at each point $x\in M$, 
 if there exists an isometry $\sigma_{\xi}$ of $\tilde{M}$ which satisfies
\begin{align*}
 \sigma_{\xi}(x)=x,\  
 (d\sigma_{\xi})_{x}(\xi)=-\xi,
 \sigma_{\xi}(M)=M,\
\end{align*}
 then we call $M$ a weakly reflective submanifold of $\tilde{M}$.
  \end{defn}
\begin{defn}[\cite{Taketomi}]
Let $M$ be a submanifold of a Riemannian manifold $(\tilde{M}, \langle , \rangle)$.
For each nonzero normal vector $\xi \in T^{\perp}_{x}M \setminus \{0\}$ at each point $x\in M$, 
if there exists an isometry $\sigma_{\xi}$ of $\tilde{M}$ which satisfies 
\begin{align*}
 \sigma_{\xi}(x)=x,\  (d\sigma_{\xi})_{x}(\xi)\neq \xi, \sigma_{\xi}(M)=M,\
\end{align*}
then we call $M$ an arid submanifold of $\tilde{M}$.  
 \end{defn}
 In order for a submanifold to be reflective, weakly reflective and arid, 
 special isometry must be exists.
 This means that these properties indicate global symmetry of submanifolds.
 
 It is known that a totally geodesic orbit of Hermann action is a reflective submanifold.
 However, austere orbits and minimal orbits of Hermann actions  is not necessary weakly reflective and arid, respectively.
In  \cite{Ohno1}, 
when $(\theta_{1}\theta_{2})^{2}=\mathrm{id}$, we examine a sufficient condition for orbits of Hermann actions to be weakly reflective in terms of symmetric triad with multiplicities.

In this section, 
we express a sufficient conditions for orbits to be weakly reflective or arid
using the contents of Section~\ref{sect-orbitspace}.

We consider the orbit $K_{2}p_{1}\subset M_{1}$ for $H\in \mathfrak{a}$.  Set $g=\exp (H),\ p_{1}=\pi_{1}(g)$.
For 
$\varepsilon \in \mathrm{U}(1)$, 
we set 
\begin{align*}
\Sigma_{\varepsilon, H} =\{\alpha \in \Sigma_{\varepsilon} \mid \langle \alpha, H \rangle +\varphi_{\varepsilon} \in \pi\mathbb{Z}\}
\end{align*}
and 
\begin{align*}
\tilde{\Sigma}_{H}=\bigcup_{\varepsilon \in \mathrm{U}(1)} \Sigma_{\varepsilon, H}.
\end{align*}
Then we have the following proposition. 
\begin{prop}\label{prop-root-system}
Let $H\in \mathfrak{a}$.
If $\tilde{\Sigma}_{H}\neq \emptyset$, 
then $\tilde{\Sigma}_{H}$ is a root system of $\mathrm{Span}(\tilde{\Sigma}_{H}) \subset \mathfrak{a}$.
\end{prop}
\begin{proof}
We consider the orthogonal symmetric Lie algebra 
\begin{align*}
(\mathfrak{k}_{1} \cap \mathrm{Ad}(g)^{-1}\mathfrak{k}_{2} )
\oplus 
(\mathfrak{m}_{1} \cap \mathrm{Ad}(g)^{-1}\mathfrak{m}_{2} ).
\end{align*}
By Lemma~\ref{lem-I-root}, we can decompose the Lie algebra as the following:
\begin{align*}
\left( V(0, 1) \oplus \sum_{\varepsilon \in \mathrm{U}(1)}\sum_{\alpha \in \Sigma_{\varepsilon}^{+} \atop \langle \alpha, H\rangle +\varphi_{\varepsilon} \in \pi\mathbb{Z} }  (V(\alpha, \varepsilon )) \cap \mathfrak{k}_{1})\right)\\
\oplus \left( \mathfrak{a}  \oplus \sum_{\varepsilon \in \mathrm{U}(1)}\sum_{\alpha \in \Sigma_{\varepsilon}^{+} \atop \langle \alpha, H\rangle +\varphi_{\varepsilon} \in \pi\mathbb{Z} }  (V(\alpha, \varepsilon )) \cap \mathfrak{m}_{1})\right).
\end{align*}
It is the root space decomposition of the orthogonal symmetric Lie algebra with respect to $\mathfrak{a}$.
\end{proof}
\begin{lem}\label{lem-root-reflection}
Let $g=\exp(H)\ (H\in \mathfrak{a})$.
Then for each $\lambda \in \tilde{\Sigma}_{H}$, 
there exists $k_{\lambda }\in \mathrm{N}_{K_{2}}(\mathfrak{a})$, 
such that 
\begin{enumerate}
\item $k_{\lambda} \pi_{1}(g)=\pi_{1}(g)$
\item $(dL_{k_{\lambda}})_{\pi_{1}(g)} (dL_{g} \xi)=dL_{g}(s_{\lambda} \xi)\ \ (\xi \in \mathfrak{a})$
\end{enumerate}
\end{lem}
\begin{proof}
By the definition of $W(\tilde{\Sigma}, \{\Sigma_{\varepsilon}\}_{\varepsilon \in \mathrm{U}(1)})$,
for each $\lambda \in \tilde{\Sigma}_{H}$, 
\begin{align*}
\left( s_{\lambda}, 2\frac{\langle \lambda, H\rangle }{\langle \lambda, \lambda \rangle} \lambda \right) \in W(\tilde{\Sigma}, \{\Sigma_{\varepsilon}\}_{\varepsilon \in \mathrm{U}(1)})
\end{align*}
Since $ W(\tilde{\Sigma}, \{\Sigma_{\varepsilon}\}_{\varepsilon \in \mathrm{U}(1)}) \subset \tilde{J}$, 
there exists $k_{\lambda } \in \mathrm{N}_{K_{2}}(\mathfrak{a})$ such that 
\begin{align*}
\left( [k_{\lambda }] , 2\frac{\langle \lambda , H\rangle }{\langle \lambda , \lambda \rangle } \lambda \right)
=\left( s_{\lambda} , 2\frac{\langle \lambda , H\rangle }{\langle \lambda , \lambda \rangle } \lambda \right).
\end{align*}
By the defnition of $\tilde{J}$, we have 
\begin{align*}
\exp \left( -2\frac{\langle \lambda , H\rangle }{\langle \lambda , \lambda \rangle } \lambda\right) k_{\lambda } \in K_{1}.
\end{align*}

For 1, 
\begin{align*}
k_{\lambda}\cdot \pi_{1}(g)
=&\pi_{1}\left( k_{\lambda } \exp (H) k_{\lambda}^{-1} \exp \left( 2\frac{\langle \lambda , H\rangle }{\langle \lambda , \lambda \rangle } \lambda \right)\right)\\
=&\pi_{1}\left( \exp \left(\mathrm{Ad}(k_{\lambda})H +2\frac{\langle \lambda , H\rangle }{\langle \lambda , \lambda \rangle } \lambda \right)\right)\\
=&\pi_{1}\left( \exp \left(s_{\lambda}(H) +2\frac{\langle \lambda , H\rangle }{\langle \lambda , \lambda \rangle } \lambda \right)\right)\\
=&\pi_{1}(\exp (H))=\pi_{1}(g).
\end{align*}

For 2, 
\begin{align*}
(dL_{k_{\lambda}})_{\pi_{1}(g)}(dL_{g} \xi)
=&\left. \frac{d}{dt} k_{\lambda}\pi_{1}(g\exp(t\xi )) \right|_{t=0}\\
=&\left. \frac{d}{dt}  \pi_{1}(k_{\lambda } \exp (H+t\xi)) \right|_{t=0}\\
=&\left. \frac{d}{dt}  \pi_{1}( \exp (\mathrm{Ad}(k_{\lambda}) (H+t\xi) ) k_{\lambda}) \right|_{t=0}\\
=&\left. \frac{d}{dt}  \pi_{1}( \exp (s_{\lambda} (H+t\xi) ) k_{\lambda}) \right|_{t=0}\\
=&\left. \frac{d}{dt}  \pi_{1}( \exp (H -2\frac{\langle \lambda , H \rangle}{\langle \lambda , \lambda \rangle }\lambda +ts_{\lambda}(\xi)) k_{\lambda}) \right|_{t=0}\\
=&\left. \frac{d}{dt}  \pi_{1}( \exp (H) \exp(ts_{\lambda}(\xi)) \exp( -2\frac{\langle \lambda , H \rangle}{\langle \lambda , \lambda \rangle }\lambda )  k_{\lambda}) \right|_{t=0}\\
=&dL_{g}d\pi_{1}(s_{\lambda}(\xi))=dL_{g}(s_{\lambda}(\xi))
\end{align*}
holds for $\xi \in \mathfrak{a}$.
\end{proof}
\begin{thm}\label{thm-arid-weakly refrective}
Let $g =\exp (H) \ (H \in \mathfrak{a})$.
If $\mathrm{Span}(\tilde{\Sigma}_{H})=\mathfrak{a}$, 
then $K_{2}p_{1}\subset M_{1}$ is an arid submanifold.
Moreover, if $\mathrm{Span}(\tilde{\Sigma}_{H})=\mathfrak{a}$
and $-\mathrm{id}_{\mathfrak{a}} \in W(\tilde{\Sigma}_{H})$, 
then $K_{2}p_{1}\subset M_{1}$ is a weakly reflective submanifold.
\end{thm}
\begin{proof}
From (\ref{eq-slicerep2}), 
we can assume $\xi \in \mathfrak{a}$.
Now, we suppose $\mathrm{Span}(\tilde{\Sigma}_{H})=\mathfrak{a}$.
From Proposition~\ref{prop-root-system}, 
The Weyl group $W(\tilde{\Sigma}_{H})$ of $\tilde{\Sigma}_{H}$ acts on $\mathfrak{a}$.
Since $\mathrm{Span}(\tilde{\Sigma}_{H})=\mathfrak{a}$, 
a vector which is invariant under the action of $W(\tilde{\Sigma}_{H})$ is only $0$. 

Then, from Lemma~\ref{lem-root-reflection}, 
for each $\sigma \in W(\tilde{\Sigma}_{H})$, 
there exists $k_{\sigma}\in K_{2}$,  
$k_{\sigma}\pi_{1}(g)=\pi_{1}(g)$ and 
$(dL_{k_{\sigma}})_{\pi_{1}(g)}(dL_{g}(\xi))=dL_{g}(\sigma(\xi))$ holds. 
Thus  
the orbit $K_{2}\pi_{1}(g)\subset M_{1}$ is a arid submanifold of $M_{1}$

Moreover, if $-\mathrm{id}_{\mathfrak{a}} \in W(\tilde{\Sigma}_{H})$, 
then there exists  $k\in K_{2}$ such that 
$(dL_{k_{\sigma}})_{\pi_{1}(g)}(dL_{g}(\xi))=-dL_{g}(\xi)$ holds 
for all $\xi \in \mathfrak{a}$.
Therefore, 
$K_{2}\pi_{1}(g)\subset M_{1}$ is a weakly  reflective submanifold of $M_{1}$.
\end{proof}
It is known that 
the following proposition.
\begin{prop}[\cite{Tits}] \label{prop-tits--1}
Let $\Sigma$ be an irreducible root system of $\mathfrak{a}$.
then, $-\mathrm{id}_{\mathfrak{a}} \not\in W(\Sigma)$
if and only if 
$\Sigma \cong \mathrm{A}_{r},\ \mathrm{D}_{2r+1}, E_{6}\ (r\geq 2)$.
\end{prop}
\section{List of non-commutative Hermann actions.}\label{list}
In this section, we introduce examples of compact symmetric triads.
Matsuki (\cite{M2}) classified pairs of involutions of compact semisimple Lie algebra under certain equivalence relation.

Here we introduce the examples whose involutions are not commutative from Matsuki's list.

Cases of $\mathfrak{g}$ is simple and $\dim \mathfrak{a} >0$.
\begin{table}[htb]
  \begin{tabular}{|c|c|c|c|c|c|} \hline
    $\mathfrak{g}$      & $\mathfrak{k}_{2}$ & $\mathfrak{k}_{2}$ & $\mathrm{dim}(\mathfrak{a})$ &$\tilde{\Sigma}$ & $l$\\ \hline \hline
    $\mathfrak{o}(m)$ & $\mathfrak{o}(p)\oplus \mathfrak{o}(q)$  & $\mathfrak{u}(m)$ & $[q/2] $ & $\mathrm{BC}_{[q/2]}$& 4\\ 
    &($p$: odd, \ $p>q$) & & & &  \\ \hline
    $\mathfrak{su}(m)$ & $\mathfrak{s}( \mathfrak{u}(p)\oplus \mathfrak{u}(q)) $  & $\mathfrak{sp}(m)$ & $[q/2] $ &$\mathrm{BC}_{[q/2]}$& 4\\ 
    &($p$: odd, \ $p>q$) & & &&  \\ \hline
    $\mathfrak{o}(8)$ & $\mathfrak{o}(5)\oplus \mathfrak{o}(3) $  & $\kappa (\mathfrak{o}(5)\oplus \mathfrak{o}(3) )$ & $2 $ & $\mathrm{G}_{2}$& 3\\ \hline
   \end{tabular}
\end{table}

Here we set 
$
I_{a, b}=
\left[
\begin{array}{c|c}
\mathrm{id}_{\mathbb{R}^{a}}&0\\ \hline
0&-\mathrm{id}_{\mathbb{R}^{b}}
\end{array}
\right]$ 
 and  $\kappa$ denotes the outer automorphism of $\mathfrak{o}(8)$ which satisfies 
$\kappa \mathrm{Ad}(I_{4,4})=\mathrm{Ad}(I_{4,4}) \kappa$ and
$\mathrm{Ad}(I'_{7,1})\kappa \mathrm{Ad}(I'_{7,1})= \kappa^{-1}$.
Here $I'_{7,1}=I_{4,4} \times I_{5,3}$.
$m\geq 3,\ p+q=m$.
$l$ denotes the order of $\theta_{1}\theta_{2} $.

\section{Weakly reflective orbits and Austere orbits}\label{classify}
In this section, we consider weakly reflective orbits and austere orbits of Hermann actions whose involutions are not commutative.
\subsection{$(\mathrm{SO}(2m), \mathrm{SO}(p)\times \mathrm{SO}(q), \mathrm{U}(m))$}
We consider the cases of $(G, K_{1},K_{2})=(\mathrm{SO}(2m), \mathrm{SO}(p)\times \mathrm{SO}(q), \mathrm{U}(m))$
($p+q=2m, p>q \geq 3,\ q:{\text{odd}}$ ). 

Then, 
the order of $\theta_{1}\theta_{2}$ is equals to $4$ and  
$\dim \mathfrak{a}=r =\frac{q-1}{2}$ holds. 
From Matsuki's list (\cite{M2}), 
we have 
\begin{align*}
\tilde{\Sigma}=&\{ \pm e_{i}\pm e_{j} \mid 1\leq i<j\leq r\} \cup \{\pm e_{i} \mid 1\leq i\leq r\} \cup \{\pm 2 e_{i} \mid 1\leq i\leq r\} \\
\cong& \mathrm{BC}_{r}\\
\Sigma_{1}   =&\{ \pm e_{i}\pm e_{j} \mid 1\leq i<j\leq r\} \cup \{\pm e_{i} \mid 1\leq i\leq r\} \cup \{\pm 2 e_{i} \mid 1\leq i\leq r\} \\ 
\cong& \mathrm{BC}_{r}\\
\Sigma_{-1}  =&\{ \pm e_{i}\pm e_{j} \mid 1\leq i<j\leq r\} \cup \{\pm e_{i} \mid 1\leq i\leq r\} 
\cong  \mathrm{B}_{r}\\
\Sigma_{\pm \sqrt{-1}}  =& \{\pm e_{i} \mid 1\leq i\leq r\} 
\cong  (\mathrm{A}_{1})^{r}, 
\end{align*}
$m(\pm e_{i}\pm e_{j}, \pm 1)=2, m(\pm e_{i}, \pm 1)=p-q, m(\pm e_{i}, \pm \sqrt{-1})=2, m(\pm 2 e_{i}, \pm 1)=1$. 

From Section~\ref{sect-orbitspace}, we have
\begin{align*}
P_{0}=&\left\{ H \in \mathfrak{a} \mid 0<\langle \alpha, H  \rangle  <\pi \ (\alpha \in \Sigma_{1}^{+})\right\} \\
&\cup \left\{ H \in \mathfrak{a} \mid -\frac{\pi}{2}<\langle \alpha, H  \rangle  <\frac{\pi}{2} \ (\alpha \in \Sigma_{-1}^{+})\right\}\\
&\cup \left\{ H \in \mathfrak{a} \mid -\frac{\pi}{4}<\langle \alpha, H  \rangle  <\frac{3 \pi}{4} \ (\alpha \in \Sigma_{\sqrt{-1}}^{+})\right\}\\
&\cup \left\{ H \in \mathfrak{a} \mid -\frac{3\pi}{4}<\langle \alpha, H  \rangle  <\frac{\pi}{4} \ (\alpha \in \Sigma_{-\sqrt{-1}}^{+})\right\}.
\end{align*} 
We set 
\begin{align*}
\tilde{\Pi}=\{\alpha_{1}=e_{1}-e_{2},\ldots , \alpha_{r-1}=e_{r-1}-e_{r}, \alpha_{r}=e_{r} \}.
\end{align*}
For each $i\in \{1, \ldots r\}$, we define $H_{i}$ by  
$\langle H_{i}, \alpha_{j}\rangle =\delta_{i, j} \ (j\in \{1,\ldots, r\})$.

Then for each $H\in \mathfrak{a}$, 
we can write $H=\sum_{i=1}^{r}x_{i}H_{i}$. 
Then $P_{0}$ described as 
\begin{align*}
P_{0}=\left\{H\in \mathfrak{a} \mid 0< x_{i}\ (i\in \{1, \ldots r\}),\ x_{1}+\cdots +x_{r} <\frac{\pi}{4} \right\}.
\end{align*}

Let we set $\tilde{\alpha}=e_{1}=\alpha_{1}+\cdots +\alpha_{r}$ and 
 
\begin{align*}
P_{0}^{\Delta}=
\left\{
\begin{array}{c|c}
 & \langle \lambda, H \rangle >0 \ (\lambda \in \Delta \cap \tilde{\Pi}) \\
H\in\mathfrak{a}  & \langle \mu , H \rangle =0 \ (\mu \in \tilde{\Pi} \setminus  \Delta  ) \\
& \langle \tilde{\alpha } , H \rangle \begin{cases} \leq \frac{\pi}{4}  \ (\ \text{if}\  \tilde{\alpha } \in \Delta )\\= \frac{\pi}{4} \ (\ \text{if}\  \tilde{\alpha } \not\in \Delta )\end{cases} 
 \end{array}
\right\}
\end{align*}
for a subset $\Delta \subset \tilde{\Pi} \cup \{\tilde{\alpha }\}$.
Then the closure $\overline{P_{0}}$ of $P_{0}$ is decomposed as 
\begin{align*}
\bigcup_{\Delta \subset \tilde{\Pi} \cup \{\tilde{\alpha}\}} P_{0}^{\Delta}\ \ (\text{disjoint union}).
\end{align*}

For each $H\in P_{0}^{\Delta}$, 
$\mathfrak{a}=\mathrm{Span} (\tilde{\Sigma}_{H})$ and $\Delta$ is a set of one point are equivalent.
Hence, from Theorem~\ref{thm-arid-weakly refrective}, 
for each $\alpha \in \tilde{\Pi} \cup \{\tilde{\alpha}\}$, 
the orbit $K_{2}p_{1}\subset M_{1}$ corresponding $H\in P_{0}^{\{\alpha \}}$ is arid. 

(0)\ When $H\in P_{0}^{\{\tilde{\alpha}\}}$, $H=0$ and   
$\tilde{\Sigma}_{H}=\Sigma_{1}\cong \mathrm{BC}_{r}$ holds.

(1) For $i\in \{1, 2, \ldots, r\}$, 
when $H\in P_{0}^{\{\alpha_{i}\}}$, $H=\frac{\pi}{4}H_{i}$ holds.
Then we have 
\begin{align*}
\Sigma_{1, H}^{+}=&\{ e_{s}-e_{t} \mid 1\leq s<t \leq i \} \cup \{e_{s}\pm e_{t} \mid 1+i\leq s<t\leq r\} \\
&\cup \{ e_{s} \mid 1+i\leq s<t\leq r \} \cup \{ 2e_{s} \mid 1+i\leq s\leq r \}, \\
\Sigma_{-1, H}^{+}=&\{ e_{s}+e_{t} \mid 1\leq s<t \leq i \},\\
\Sigma_{\sqrt{-1}, H}^{+}=&\emptyset, \\
\Sigma_{-\sqrt{-1}, H}^{+}=&\{ e_{s} \mid 1\leq s \leq i \}.
\end{align*}
Therefore, we have 
$\tilde{\Sigma}_{H} \cong \mathrm{B}_{i}\oplus \mathrm{BC}_{r-i}$. 

From Theorem~\ref{thm-arid-weakly refrective} and Proposition~\ref{prop-tits--1}, 
When $H=0, \frac{\pi}{4}H_{1}, \ldots, \frac{\pi}{4}H_{r}$, 
the orbit $K_{2}p_{1}\subset M_{1}$ is weakly reflective. 

Next, we consider austere orbits.

For $H\in \overline{P_{0}}$, a necessary and sufficient condition for the orbit $K_{2}p_{1}\subset M_{1}$ to be austere 
is given in Proposition~\ref{prop-austere}.
If $E_{H}$ is invariant under the multiplication by $-1$,
then 
$\mathbb{R} \cdot \alpha \cap E_H$ is invariant under the multiplication by $-1$
for any $\alpha \in \tilde{\Sigma}^{+}$.  

For each $i \in \{ 1, \ldots, r-1\} $, when $x_{i}\not\in \frac{\pi}{2}\mathbb{Z}$, we have  
\begin{align*}
\mathbb{R} \cdot \alpha_{i} \cap E_H=\left\{ - \cot x_{i} \alpha_{i} \ (\text{multiplicity}\ 2 ), 
  -\cot (x_{i}+\frac{\pi}{2}) \alpha_{i} \ (\text{multiplicity}\ 2 ) \right\}.
\end{align*}
When $\mathbb{R} \cdot \alpha_{i} \cap E_H$ is invariant under the multiplication by $-1$,  
$x_{i}=0$ or $\frac{\pi}{4}$ holds.
Also,  when $x_{r} \not\in \frac{\pi}{4}\mathbb{Z}$ we have 
\begin{align*}
\mathbb{R} \cdot \alpha_{r} \cap E_H
=\{ & - \cot x_{r} \alpha_{r} \ (\text{multiplicity}\ p-q ),
- \cot (2x_{r}) 2 \alpha_{r} \ (\text{multiplicity}\ 1 ), \\
&-\cot (x_{r}+\frac{\pi}{2}) \alpha_{r} \ (\text{multiplicity}\ p-q ),
-\cot (x_{r}+\frac{\pi}{4}) \alpha_{r} \ (\text{multiplicity}\ 2 ) ,\\
&-\cot (x_{r}-\frac{\pi}{4}) \alpha_{r} \ (\text{multiplicity}\ 2 ) 
  \}.
\end{align*}
When $\mathbb{R} \cdot \alpha_{r} \cap E_H$ is invariant under the multiplication by $-1$,  
$x_{r}=0$ or $\frac{\pi}{4}$.

Therefore, 
when the orbit $K_{2}p_{1}\subset M_{1}$ is austere,  $H=0, \frac{\pi}{4}H_{1}, \ldots, \frac{\pi}{4}H_{r}$ holds. 
Then the orbit $K_{2}p_{1}\subset M_{1}$ is weakly reflective. 

\subsection{$(\mathrm{SU}(2m), \mathrm{S}(\mathrm{U}(p)\times \mathrm{U}(q)), \mathrm{Sp}(m))$}
We consider the cases of $(G, K_{1},K_{2})=(\mathrm{SU}(2m), \mathrm{S}(\mathrm{U}(p)\times \mathrm{U}(q)), \mathrm{Sp}(m))$
($p+q=2m, p>q \geq 3,\ q:{\text{odd}}$ ).

Then,   
the order of $\theta_{1}\theta_{2}$ is equals to $4$ and  
$\dim \mathfrak{a}=r =\frac{q-1}{2}$ holds. 
From Matsuki's list (\cite{M2}), 
we have 
\begin{align*}
\tilde{\Sigma}=&\{ \pm e_{i}\pm e_{j} \mid 1\leq i<j\leq r\} \cup \{\pm e_{i} \mid 1\leq i\leq r\} \cup \{\pm 2 e_{i} \mid 1\leq i\leq r\} \\
\cong& \mathrm{BC}_{r}\\
\Sigma_{1}   =&\{ \pm e_{i}\pm e_{j} \mid 1\leq i<j\leq r\} \cup \{\pm e_{i} \mid 1\leq i\leq r\} \cup \{\pm 2 e_{i} \mid 1\leq i\leq r\} \\ 
\cong& \mathrm{BC}_{r}\\
\Sigma_{-1}  =&\{ \pm e_{i}\pm e_{j} \mid 1\leq i<j\leq r\} \cup \{\pm e_{i} \mid 1\leq i\leq r\} \cup \{\pm 2 e_{i} \mid 1\leq i\leq r\} \\ 
\cong& \mathrm{BC}_{r}\\
\Sigma_{\pm \sqrt{-1}}  =& \{\pm e_{i} \mid 1\leq i\leq r\} 
\cong  (\mathrm{A}_{1})^{r}, 
\end{align*}
$m(\pm e_{i}\pm e_{j}, \pm 1)=4, m(\pm e_{i}, \pm 1)=2(p-q), m(\pm e_{i}, \pm \sqrt{-1})=4, m(\pm 2 e_{i},  1)=3,m(\pm 2 e_{i},  -1)=1$.

We set 
\begin{align*}
\tilde{\Pi}=\{\alpha_{1}=e_{1}-e_{2},\ldots , \alpha_{r-1}=e_{r-1}-e_{r}, \alpha_{r}=e_{r} \}.
\end{align*}
For each $i\in \{1, \ldots r\}$, we define $H_{i}$ by  
$\langle H_{i}, \alpha_{j}\rangle =\delta_{i, j} \ (j\in \{1,\ldots, r\})$.

Then for each $H\in \mathfrak{a}$, 
we can write $H=\sum_{i=1}^{r}x_{i}H_{i}$. 
Then $P_{0}$ described as 
\begin{align*}
P_{0}=\left\{H\in \mathfrak{a} \mid 0< x_{i}\ (i\in \{1, \ldots r\}),\ x_{1}+\cdots +x_{r} <\frac{\pi}{4} \right\}.
\end{align*}
 
Let we set $\tilde{\alpha}=e_{1}=\alpha_{1}+\cdots +\alpha_{r}$ and 
\begin{align*}
P_{0}^{\Delta}=
\left\{
\begin{array}{c|c}
 & \langle \lambda, H \rangle >0 \ (\lambda \in \Delta \cap \tilde{\Pi}) \\
H\in\mathfrak{a}  & \langle \mu , H \rangle =0 \ (\mu \in \tilde{\Pi} \setminus  \Delta  ) \\
& \langle \tilde{\alpha } , H \rangle \begin{cases} \leq \frac{\pi}{4}  \ (\ \text{if}\  \tilde{\alpha } \in \Delta )\\= \frac{\pi}{4} \ (\ \text{if}\  \tilde{\alpha } \not\in \Delta )\end{cases} 
 \end{array}
\right\}
\end{align*}
for a subset $\Delta \subset \tilde{\Pi} \cup \{\tilde{\alpha }\}$.
Then the closure $\overline{P_{0}}$ of $P_{0}$ is decomposed as 
\begin{align*}
\bigcup_{\Delta \subset \tilde{\Pi} \cup \{\tilde{\alpha}\}} P_{0}^{\Delta}\ \ (\text{disjoint union}).
\end{align*}

For each $H\in P_{0}^{\Delta}$, 
$\mathfrak{a}=\mathrm{Span} (\tilde{\Sigma}_{H})$ and $\Delta$ is a set of one point are equivalent.
Hence, from Theorem~\ref{thm-arid-weakly refrective}, 
for each $\alpha \in \tilde{\Pi} \cup \{\tilde{\alpha}\}$, 
the orbit $K_{2}p_{1}\subset M_{1}$ corresponding $H\in P_{0}^{\{\alpha \}}$ is arid.

(0)\ When $H\in P_{0}^{\{\tilde{\alpha}\}}$, $H=0$ and   
$\tilde{\Sigma}_{H}=\Sigma_{1}\cong \mathrm{BC}_{r}$ holds.

(1) For $i\in \{1, 2, \ldots, r\}$, 
when $H\in P_{0}^{\{\alpha_{i}\}}$, $H=\frac{\pi}{4}H_{i}$ holds.
Then we have 
\begin{align*}
\Sigma_{1, H}^{+}=&\{ e_{s}-e_{t} \mid 1\leq s<t \leq i \} \cup \{e_{s}\pm e_{t} \mid 1+i\leq s<t\leq r\} \\
&\cup \{ e_{s} \mid 1+i\leq s<t\leq r \} \cup \{ 2e_{s} \mid 1+i\leq s\leq r \}, \\
\Sigma_{-1, H}^{+}=&\{ e_{s}+e_{t} \mid 1\leq s<t \leq i \} \cup \{ 2e_{s} \mid 1\leq s\leq s \},\\
\Sigma_{\sqrt{-1}, H}^{+}=&\emptyset, \\
\Sigma_{-\sqrt{-1}, H}^{+}=&\{ e_{s} \mid 1\leq s \leq i \}.
\end{align*}
Therefore, we have 
$\tilde{\Sigma}_{H} \cong \mathrm{BC}_{i}\oplus \mathrm{BC}_{r-i}$. 

From Theorem~\ref{thm-arid-weakly refrective} and Proposition~\ref{prop-tits--1}, 
When $H=0, \frac{\pi}{4}H_{1}, \ldots, \frac{\pi}{4}H_{r}$, 
the orbit $K_{2}p_{1}\subset M_{1}$ is weakly reflective. 

Next, we consider austere orbits.

If $E_{H}$ is invariant under the multiplication by $-1$,
then 
$\mathbb{R} \cdot \alpha \cap E_H$ is invariant under the multiplication by $-1$
for any $\alpha \in \tilde{\Sigma}^{+}$.  
For each $i \in \{ 1, \ldots, r-1\} $, when $x_{i}\not\in \frac{\pi}{2}\mathbb{Z}$, we have  

\begin{align*}
\mathbb{R} \cdot \alpha_{i} \cap E_H=\left\{ - \cot x_{i} \alpha_{i} \ (\text{multiplicity}\ 4 ), 
  -\cot (x_{i}+\frac{\pi}{2}) \alpha_{i} \ (\text{multiplicity}\ 4 ) \right\}
\end{align*}

When $\mathbb{R} \cdot \alpha_{i} \cap E_H$ is invariant under the multiplication by $-1$,  
$x_{i}=0$ or $\frac{\pi}{4}$.
Also,  when $x_{r} \not\in \frac{\pi}{4}\mathbb{Z}$ we have 
\begin{align*}
&\mathbb{R} \cdot \alpha_{r} \cap E_H\\
=\{ & - \cot x_{r} \alpha_{r} \ (\text{multiplicity}\ 2(p-q) ),
- \cot (2x_{r}) 2 \alpha_{r} \ (\text{multiplicity}\ 3 ), \\
&-\cot (x_{r}+\frac{\pi}{2}) \alpha_{r} \ (\text{multiplicity}\ 2(p-q) ),
- \cot (2x_{r}+\frac{\pi}{2}) 2 \alpha_{r} \ (\text{multiplicity}\ 1 ), \\
&-\cot (x_{r}+\frac{\pi}{4}) \alpha_{r} \ (\text{multiplicity}\ 4 ) ,
-\cot (x_{r}-\frac{\pi}{4}) \alpha_{r} \ (\text{multiplicity}\ 4 ) 
  \}.
\end{align*}
Then, at least one of the following equations holds:
\begin{align*}
\begin{cases}
\cot x_{r}=\tan x_{r}\\
\cot x_{r}=\tan (x_{r}+\frac{\pi}{4})\\
\cot x_{r}=2\tan 2x_{r}.\\
\end{cases}
\end{align*}
Thus, we have $x_{r}=\frac{\pi}{4}, \frac{\pi}{8}$ or $(\cot x_{r})^{2}=5$. 
By a simple calculation shows that 
if $x_{r}= \frac{\pi}{8}$ or $(\cot x_{r})^{2}=5$ holds, 
then $\mathbb{R} \cdot \alpha_{r} \cap E_H$ is not invariant under the multiplication by $-1$. 

Therefore, $x_{r} =0, \frac{\pi}{4}$ holds when the orbit $K_{2}p_{1}\subset M_{1}$ is austere. 

Therefore, 
when the orbit $K_{2}p_{1}\subset M_{1}$ is austere,  $H=0, \frac{\pi}{4}H_{1}, \ldots, \frac{\pi}{4}H_{r}$ holds. 
Then the orbit $K_{2}p_{1}\subset M_{1}$ is weakly reflective.

\subsection{$(\mathrm{SO}(8), \mathrm{SO}(5)\times \mathrm{SO}(3), \kappa( \mathrm{SO}(5)\times \mathrm{SO}(3))  )$}
We consider the cases of 
$(G, K_{1},K_{2})=(\mathrm{SO}(8), \mathrm{SO}(5)\times \mathrm{SO}(3), \kappa( \mathrm{SO}(5)\times \mathrm{SO}(3))  )$.

Then,   
the order of $\theta_{1}\theta_{2}$ is equals to $3$ and  
$\dim \mathfrak{a}=2$ holds. 
From Matsuki's list (\cite{M2}), 
we have 
\begin{align*}
\tilde{\Sigma}\cong  \mathrm{G}_{2},
\Sigma_{1}   \cong\mathrm{G}_{2}, 
\Sigma_{ \omega}= \{ \text{Short roots} \} \cong \mathrm{A}_{2}, 
\Sigma_{\omega^{-1}}=\{\text{Short roots} \} \cong \mathrm{A}_{2},
\end{align*}
$m(\alpha , \pm 1)=1\  (\alpha \in \Sigma_{1}), m(\beta ,  \omega)=1\  (\beta \in \Sigma_{ \omega}), m(\beta ,  \omega^{-1})=1 \ (\beta \in \Sigma_{ \omega^{-1}}) $. 
Here, $\omega$ denotes the primitive third root of unity.

We set 
\begin{align*}
\tilde{\Pi}=\{\alpha_{1}, \alpha_{2} \}.
\end{align*}
Then 
\begin{align*}
\tilde{\Sigma}^{+}=\{ \alpha_{1}, \alpha_{2}, \alpha_{1}+ \alpha_{2}, 2\alpha_{1}+ \alpha_{2}, 3\alpha_{1}+ \alpha_{2}, 3\alpha_{1}+ 2\alpha_{2}\} 
\end{align*}
and  
$\langle \alpha_{1}, \alpha_{1} \rangle=2, \langle \alpha_{1}, \alpha_{2} \rangle=-1 , \langle \alpha_{2}, \alpha_{2} \rangle=6$
holds.

For $i=1,2$, 
we define $H_{1}$ by 
$\langle H_{i}, \alpha_{j}\rangle =\delta_{i, j} \ (j=1,2)$

Then for each $H\in \mathfrak{a}$, 
we can write $H=x_{1}H_{1}+x_{2}H_{2}$. 
Then $P_{0}$ described as 
\begin{align*}
P_{0}=\left\{H\in \mathfrak{a} \mid 0< x_{i}\ (i\in \{1, 2\}),\ 2 x_{1}+ x_{2} <\frac{\pi}{3} \right\}.
\end{align*}

Let we set $\tilde{\alpha}=2\alpha_{1}+\alpha_{2}$ and 
for a subset $\Delta \subset \tilde{\Pi} \cup \{\tilde{\alpha }\}$, we set 
\begin{align*}
P_{0}^{\Delta}=
\left\{
\begin{array}{c|c}
 & \langle \lambda, H \rangle >0 \ (\lambda \in \Delta \cap \tilde{\Pi}) \\
H\in\mathfrak{a}  & \langle \mu , H \rangle =0 \ (\mu \in \tilde{\Pi} \setminus  \Delta  ) \\
& \langle \tilde{\alpha } , H \rangle \begin{cases} \leq \frac{\pi}{3}  \ (\ \text{if}\  \tilde{\alpha } \in \Delta )\\= \frac{\pi}{3} \ (\ \text{if}\  \tilde{\alpha } \not\in \Delta )\end{cases} 
 \end{array}
\right\}
\end{align*}
Then the closure $\overline{P_{0}}$ of $P_{0}$ is decomposed as 
\begin{align*}
\bigcup_{\Delta \subset \tilde{Pi} \cup \{\tilde{\alpha}\}} P_{0}^{\Delta}\ \ (\text{disjoint union}).
\end{align*}

For each $H\in P_{0}^{\Delta}$, 
$\mathfrak{a}=\mathrm{Spam} (\tilde{\Sigma}_{H})$ and $\Delta$ is a set of one point are equivalent.
Hence, from Theorem~\ref{thm-arid-weakly refrective}, 
for each $\alpha \in \tilde{\Pi} \cup \{\tilde{\alpha}\}$, 
the orbit $K_{2}p_{1}\subset M_{1}$ corresponding $H\in P_{0}^{\{\alpha \}}$ is arid. 

(0)\ When $H\in P_{0}^{\{\tilde{\alpha}\}}$, $H=0$ and 
$\tilde{\Sigma}_{H}=\Sigma_{1}\cong \mathrm{G}_{2}$.

(1) 
When $H\in P_{0}^{\{\alpha_{1}\}}$, $H=\frac{\pi}{6}H_{1}$ holds.
Then we have 
\begin{align*}
\Sigma_{1, H}^{+}=\{\alpha_{2} \},
\Sigma_{\omega, H}^{+}=\emptyset ,
\Sigma_{\omega^{-1}, H}^{+}=\{ 2\alpha_{1}+\alpha_{2}\}.
\end{align*}
Therefore, we have 
$\tilde{\Sigma}_{H} \cong \mathrm{A}_{1}\oplus \mathrm{A}_{1}$. 

(2) 
When $H\in P_{0}^{\{\alpha_{2}\}}$, $H=\frac{\pi}{3}H_{2}$ holds.
Then we have 
\begin{align*}
\Sigma_{1, H}^{+}=\{\alpha_{1} \}, 
\Sigma_{\omega, H}^{+}=\emptyset ,
\Sigma_{\omega^{-1}, H}^{+}=\{ \alpha_{1}+\alpha_{2} , 2\alpha_{1}+\alpha_{2}\} .
\end{align*}
Therefore, we have 
$\tilde{\Sigma}_{H} \cong \mathrm{A}_{2}$. 

From Theorem~\ref{thm-arid-weakly refrective} and Proposition~\ref{prop-tits--1}, 
When $H=0, \frac{\pi}{6}H_{1}$, 
the orbit $K_{2}p_{1}\subset M_{1}$ is weakly reflective. 

Next, we consider austere orbits.
If $E_{H}$ is invariant under the multiplication by $-1$,
then 
$\mathbb{R} \cdot \alpha \cap E_H$ is invariant under the multiplication by $-1$
for any $\alpha \in \tilde{\Sigma}^{+}$.  

When $x_{1}\not\in \frac{\pi}{3}\mathbb{Z}$, 
we have
\begin{align*}
\mathbb{R} \cdot \alpha_{1} \cap E_H
=\left\{ - \cot x_{1} \alpha_{1} \ (\text{multiplicity}\ 1 ), 
  -\cot (x_{1}+\frac{\pi}{3}) \alpha_{1} \ (\text{multiplicity}\ 1 ), \right. \\
    \left.-\cot (x_{1}-\frac{\pi}{3}) \alpha_{1} \ (\text{multiplicity}\ 1 ) \right\}.
\end{align*}
Thus, 
When $\mathbb{R} \cdot \alpha_{1} \cap E_H$ is invariant under the multiplication by $-1$,  
$x_{1}=0$ or $\frac{\pi}{6}$.

When $x_{2}\not\in \frac{\pi}{3}\mathbb{Z}$, 
we have
\begin{align*}
\mathbb{R} \cdot \alpha_{r} \cap E_H
=\{  - \cot x_{2} \alpha_{2} \ (\text{multiplicity}\ 1 ),  \}.
\end{align*}
Thus we have $x_{2}=0$. 

Therefore, 
when the orbit $K_{2}p_{1}\subset M_{1}$ is austere,  $H=0, \frac{\pi}{6}H_{1}$ holds. 
Then the orbit $K_{2}p_{1}\subset M_{1}$ is weakly reflective. 

If $H=\frac{\pi}{3}H_{2}$, then the orbit $K_{2}p_{1}\subset M_{1}$ is an arid orbit which is not austere. 


\end{document}